\documentclass[reqno,12pt]{amsart}

\usepackage{bbm}

\numberwithin{equation}{section}

\newcommand{\1}{\mathbbm {1}}
\newcommand{\Z}{{\mathbb Z}}
\newcommand{\C}{{\mathbb C}}

\newcommand{\wh}{{\widehat{\mathfrak h}}}
\newcommand{\I}{{\mathcal I}}
\newcommand{\tI}{\widetilde{\I}}
\newcommand{\CA}{{\mathcal A}}
\newcommand{\CE}{{\mathcal E}}

\newcommand{\J}{{\mathcal J}}
\newcommand{\CL}{{\mathcal L}}
\newcommand{\CP}{{\mathcal P}}
\newcommand{\W}{{\mathcal W}}
\newcommand{\tW}{\widetilde{\W}}

\newcommand{\mraff}{\mathrm{aff}}

\newcommand{\al}{\alpha}
\newcommand{\be}{\beta}
\newcommand{\dl}{\delta}
\newcommand{\gm}{\gamma}
\newcommand{\om}{\omega}

\newcommand{\la}{\langle}
\newcommand{\ra}{\rangle}
\newcommand{\bu}{\mathbf{u}}
\newcommand{\bv}{\mathbf{v}}

\newcommand{\hsl}{\widehat{sl}_2}
\newcommand{\tvp}{\tilde{\varphi}}
\newcommand{\Vaff}{V^{\mathrm{aff}}}
\newcommand{\Vh}{V_{\wh}}

\DeclareMathOperator{\Aut}{Aut}
\DeclareMathOperator{\ch}{ch}
\DeclareMathOperator{\End}{End}
\DeclareMathOperator{\Res}{Res}
\DeclareMathOperator{\spn}{span}
\DeclareMathOperator{\Vir}{Vir}
\DeclareMathOperator{\wt}{wt}

\newtheorem{thm}{Theorem}[section]
\newtheorem{prop}[thm]{Proposition}
\newtheorem{lem}[thm]{Lemma}

\newtheorem{rmk}[thm]{Remark}

\newtheorem{conjecture}[thm]{Conjecture}
\begin{document}

\title
{$W$-algebras related to parafermion algebras}

\author[C. Dong]{Chongying Dong}
\address{Department of Mathematics, University of California, Santa Cruz,
CA 95064} \email{dong@math.ucsc.edu}

\author[C. H. Lam]{Ching Hung Lam}
\address{Department of Mathematics and National Center for Theoretical
Sciences, National Cheng Kung University, Tainan, Taiwan 701}
\email{chlam@mail.ncku.edu.tw}

\author[H. Yamada]{Hiromichi Yamada}
\address{Department of Mathematics, Hitotsubashi University, Kunitachi,
Tokyo 186-8601, Japan}
\email{yamada@math.hit-u.ac.jp}


\begin{abstract}
We study a $W$-algebra of central charge $2(k-1)/(k+2)$, $k = 2, 3, \ldots$
contained in the commutant of a Heisenberg algebra in a simple affine
vertex operator algebra $L(k,0)$ of type $A_1^{(1)}$ with level $k$. We
calculate the operator product expansions of the $W$-algebra. We also
calculate some singular vectors in the case $k \le 6$ and determine the
irreducible modules and Zhu's algebra. Furthermore, the rationality and the
$C_2$-cofiniteness are verified for such $k$.
\end{abstract}

\maketitle

\section{Introduction}
A Virasoro field and its finitely many primary fields generate a
$W$-algebra of various kinds. Those $W$-algebras are important in the study
of vertex operator algebras, for they provide many interesting examples. In
\cite{Hornfeck}, a possible structure of a $W$-algebra $W(2,3,4,5)$ with
primary fields of conformal weight $3$, $4$ and $5$ was discussed. Such an
algebra was constructed in the commutant of a Heisenberg algebra in a Weyl
module $V(k,0)$ for an affine Lie algebra $\widehat{sl}_2$ of type
$A_1^{(1)}$ with level $k$ in \cite{BEHHH}. Our main concern is the
commutant $K_0$ of a Heisenberg algebra in a simple quotient $L(k,0)$ of
$V(k,0)$, where $k$ is an integer greater than $1$. The commutant $K_0$,
including the characters of its irreducible modules, was studied in
\cite{CGT} (see \cite{GQ} also). In this paper, we study $K_0$ from a point
of view of vertex operator algebra. The central charge of $K_0$ is
$2(k-1)/(k+2)$, which coincides with the central charge of the parafermion
algebra. We refer the reader to \cite{DL} for the relationship between
$K_0$ and the parafermion algebra.

It is also known that $K_0$ appears as the commutant of a certain
subalgebra in the vertex operator algebra $V_{\sqrt{2}A_{k-1}}$ associated
with $\sqrt{2}A_{k-1}$, where $\sqrt{2}A_{k-1}$ denotes $\sqrt{2}$ times an
ordinary root lattice of type $A_{k-1}$ \cite{LY}. Such a realization of
$K_0$ leads to a natural study of $V_{\sqrt{2}A_{k-1}}$ as a module for
$K_0$. The $K_0$-module structure of $V_{\sqrt{2}A_{k-1}}$ for some special
$k$ is expected to play an important role in a better understanding of the
moonshine vertex operator algebra $V^{\natural}$ \cite{FLM}.

It is widely believed that $K_0$ is a rational and $C_2$-cofinite vertex
operator algebra. It is also anticipated that $K_0$ has exactly $k(k+1)/2$
inequivalent irreducible modules (see Conjecture \ref{Conj:W-properties}).
In this paper, we treat these subjects. The key of our arguments here is a
detailed analysis of some singular vectors. Unfortunately, we do not
succeed in describing those singular vectors explicitly for a general $k$.
Therefore, we restrict ourselves to the case $k\le 6.$ We determine Zhu's
algebra and classify the irreducible modules of $K_0$ for $k \le 6$.
Moreover, we show that $K_0$ is rational and $C_2$-cofinite for such $k$.
In the case $k \ge 3$, we show that $K_0$ is generated by a primary field
of weight $3$.

The organization of the paper is as follows. In Section \ref{Sect:W2345},
we introduce the conformal vector $\omega$ of central charge $2(k-1)/(k+2)$
and Virasoro primary vectors $W^3$, $W^4$ and $W^5$ of weight $3$, $4$ and
$5$, respectively in the commutant $N_0$ of a Heisenberg algebra in the
Weyl module $V(k,0)$ for $\widehat{sl}_2$ with level $k$, where $k$ is an
integer greater than $1$. Such a vector $W^i$, $i=3,4,5$ is unique up to a
scalar multiple. Let $W^i_n$ be a component operator, that is, the
coefficient of $x^{-n-1}$ in the vertex operator associated with $W^i$. The
vectors $W^i_n W^j$, $3 \le i \le j \le 5$, $0 \le n \le i+j-1$ are known
in \cite{BEHHH}. We compute these vectors in the Weyl module $V(k,0)$ and
express them as linear combinations of vectors of normal form (see
\eqref{eq:normal-form}). The computation has been done by a computer
algebra system Risa/Asir. The results can be found in Appendix
\ref{App:OPE}. In the computation of the vectors $W^i_n W^j$, we do not
assume that $k$ is an integer greater than $1$. Thus in Appendix
\ref{App:OPE}, we can think of the parameter $k$ as a formal variable.
Using the explicit expression of $W^i_n W^j$ as a linear combination of
vectors of normal form, we study a subalgebra $\tW$ of the commutant $N_0$
generated by $\omega$, $W^3$, $W^4$ and $W^5$. It turns out that $\tW$ is
in fact generated by $W^3$ if $k \ge 3$. As a consequence, the automorphism
group of $\tW$ is of order $2$ and it is generated by an automorphism which
maps $W^3$ to its negative if $k \ge 3$. We also show that Zhu's algebra of
$\tW$ is commutative. It is known that $\tW$ has two (resp. four) linearly
independent singular vectors of weight $8$ (resp. $9$) \cite{BEHHH}. We use
these singular vectors to determine Zhu's algebra of $K_0$ for $k = 5, 6$
in Section \ref{Sect:small-k-case}. In addition to them, a weight $10$
singular vector is necessary to establish the $C_2$-cofiniteness for $k =
5, 6$ in Section \ref{Sect:small-k-case}.

Since $k$ is an integer greater than $1$, the vertex operator algebra
$V(k,0)$ possesses a unique maximal ideal $\J$, which is generated by
$e(-1)^{k+1}\1$ \cite{K}. In Section \ref{Sect:maximal-ideal-tI}, we study
the commutant $K_0$ of a Heisenberg algebra in the quotient vertex operator
algebra $L(k,0) = V(k,0)/\J$. We denote the image of $\tW$ in $L(k,0)$ by
$\W$. Then $\W$ is a subalgebra of $K_0$. The ideal $\J$ is not contained
in the commutant $N_0$. It is expected that a unique maximal ideal $\J \cap
N_0$ of $N_0$ is generated by a weight $k+1$ vector $\bu^0 =
f(0)^{k+1}e(-1)^{k+1}\1$ (see Lemma \ref{lem:unique-max-ideal}, Conjecture
\ref{Conj:ideal-generator}).

In Section \ref{Sect:VOA-W}, we embed $L(k,0)$ into a vertex operator
algebra $V_L$ associated with a lattice $L$ of type $A_1^{\oplus k}$. This
is accomplished by the use of level-rank duality \cite[Chapter 14]{DL}. Let
$\Vaff$ be a subalgebra of $V_L$ obtained by the embedding. Then $\Vaff
\cong L(k,0)$. There is a sublattice $L'$ of $L$ isomorphic to
$\sqrt{2}A_{k-1}$ such that the vertex operator algebra $V_{L'}$ associated
with $L'$ is the commutant of the vertex operator algebra $V_{\Z\gamma}$
associated with a rank one lattice $\Z\gamma$ in $V_L$. We have $V_L
\supset \Vaff \supset V_{\Z\gamma}$ and $K_0 \cong \Vaff \cap V_{L'}$. That
is, $K_0$ is isomorphic to the commutant of $V_{\Z\gamma}$ in $\Vaff$. This
consideration has some advantages. For instance, using the representation
theory of the vertex operator algebra $V_{\Z\gamma}$, we construct a
certain family of irreducible $K_0$-modules inside $V_{L^\perp}$ and study
their properties, where $L^\perp$ is the dual lattice of $L$.

The singular vectors of weight at most $10$ in $\tW$ are calculated
explicitly for any $k$. However, this is not the case for $\bu^0$. We can
describe $\bu^0$ as a linear combination of vectors of normal form only for
a given small $k$. For this reason, we deal with only the case $k \le 6$ in
Section \ref{Sect:small-k-case}. If $k = 2$, $3$ or $4$, then $\W$ is
degenerate. In fact, it turns out that $\bu^0$ is a scalar multiple of
$W^{k+1}$ for $k = 2,3,4$. In such a case, $\W$ is isomorphic to a
well-known vertex operator algebra. Thus the main part of Section
\ref{Sect:small-k-case} is devoted to the case $k = 5,6$. We show that $\W
= K_0$ and classify its irreducible modules. Moreover, we show that $K_0$
is rational and $C_2$-cofinite. We note that $K_0$ is related to a $2A$,
$3A$, $4A$, $5A$ or $6A$ element of the Monster simple group according as
$k=2$, $3$, $4$, $5$ or $6$ (see \cite[Section 3, Appendix B]{LYY},
\cite[Section 4]{Matsuo}). In fact, this is part of the motivation of our
work.

The argument heavily depends on singular vectors $\bv^0$, $\bv^1$ and
$\bv^2$ of weight $8$, $9$ and $10$, respectively in $\tW$ and on singular
vectors $\bu^r = (W^3_1)^r\bu^0$ of weight $k+1+r$, $r = 0,1,2,3$ in $\W$.
It seems that we can take $W^3_1\bv^0$ and $(W^3_1)^2\bv^0$ in place of
$\bv^1$ and $\bv^2$, respectively. However, we do not verify it. The
importance of $\bu^0$ is clear from the degenerate case, namely, the case
$k = 2,3,4$, for $\bu^0$ is a scalar multiple of $W^3$, $W^4$ or $W^5$ in
such a case. It would be difficult to express $\bu^r$, $r = 0,1,2,3$ in
terms of $\omega$, $W^3$, $W^4$ and $W^5$ for an arbitrary $k$. We should
take a different approach for a general case.

Our notation is fairly standard \cite{FLM, LL}. Let $V$ be a vertex
operator algebra and $(M, Y_M)$ be its module. Then $Y_M(v,x) = \sum_{n \in
\Z} v_n x^{-n-1}$ is the vertex operator associated with $v \in V$. The
linear operator $v_n$ on $M$ is called a component operator. For a
subalgebra $U$ of $V$ and a subset $S$ of $M$, let $U\cdot S =
\spn\{u_nw\,|\,u \in U, w \in S, n \in \Z\}$, which is the $U$-submodule of
$M$ generated by $S$.

Part of the results in this paper was announced in \cite{DLY2}. We remark
that $N_0$ (resp. $K_0$) is denoted by $\tW$ (resp. $\W$) in \cite{DLY2}.
In this paper, we distinguish $N_0$ and $\tW$ (resp. $K_0$ and $\W$)
clearly to avoid confusion.

\section{$\tW$ and its singular vectors}\label{Sect:W2345}
Let $\{ h, e, f\}$ be a standard Chevalley basis of $sl_2$. Thus $[h,e] =
2e$, $[h,f] = -2f$, $[e,f] = h$ for the bracket and $\la h,h \ra = 2$, $\la
e,f \ra = 1$, $\la h,e \ra = \la h,f \ra = \la e,e \ra = \la f,f \ra = 0$
for the normalized Killing form. We fix an integer $k \ge 2$. Let $V(k,0) =
V_{\hsl}(k,0)$ be a Weyl module for the affine Lie algebra $\widehat{sl}_2
= sl_2 \otimes \C[t,t^{-1}] \oplus \C C$ with level $k$, that is, a Verma
module for $\widehat{sl}_2$ with level $k$ and highest weight $0$. Let $\1$
be its canonical highest weight vector, which is called the vacuum vector.
Then $sl_2 \otimes \C[t]$ acts as $0$ and $C$ acts as $k$ on $\1$. We
denote by $h(n)$, $e(n)$ and $f(n)$ the operators on $V(k,0)$ induced by
the action of $h \otimes t^n$, $e \otimes t^n$ and $f \otimes t^n$,
respectively. Thus $h(n)\1 = e(n)\1 = f(n)\1 = 0$ for $n \ge 0$ and
\begin{equation}\label{eq:affine-commutation}
[a(m), b(n)] = [a,b](m+n) + m \la a,b \ra \delta_{m+n,0}k
\end{equation}
for $a, b \in \{h,e,f\}$. The elements
\begin{equation}\label{eq:V-basis}
h(-i_1) \cdots h(-i_p) e(-j_1) \cdots e(-j_q) f(-m_1) \cdots f(-m_r)\1,
\end{equation}
$i_1 \ge \cdots \ge i_p \ge 1$, $j_1 \ge \cdots \ge j_q \ge 1$, $m_1 \ge
\cdots \ge m_r \ge 1$ form a basis of $V(k,0)$.

Let $a(x) = \sum_{n \in \Z} a(n)x^{-n-1}$ for $a \in \{h,e,f\}$ and
\begin{equation*}
a(x)_n b(x) = \Res_{x_1} \Big( (x_1 - x)^n a(x_1)b(x) - (-x+x_1)^n
b(x)a(x_1) \Big).
\end{equation*}
Then the vertex operator $Y(v,x) = \sum_{n \in \Z} v_n x^{-n-1} \in (\End
V(k,0))[[x,x^{-1}]]$ associated with $v = a^1(n_1) \cdots a^r(n_r)\1$ is
given by
\begin{equation}\label{eq:def-vertex-operator}
Y(a^1(n_1) \cdots a^r(n_r)\1,x) = a^1(x)_{n_1} \cdots a^r(x)_{n_r}1
\end{equation}
for $a^i \in \{h,e,f\}$ and $n_i \in \Z$, where $1$ denotes the identity
operator. Set
\begin{align*}
\omega_{\mraff} &= \frac{1}{2(k+2)} \Big( \frac{1}{2}h(-1)^2\1 +
e(-1)f(-1)\1 + f(-1)e(-1)\1 \Big)\\
&= \frac{1}{2(k+2)} \Big( -h(-2)\1 + \frac{1}{2}h(-1)^2\1 +
2e(-1)f(-1)\1 \Big).
\end{align*}
Then $(V(k,0), Y, \1, \omega_{\mraff})$ is a vertex operator algebra with
the conformal vector $\omega_{\mraff}$, whose central charge is $3k/(k+2)$
\cite{FZ} (see \cite[Section 6.2]{LL} also). The vector of the form
\eqref{eq:V-basis} is an eigenvector for $(\omega_{\mraff})_1$ with
eigenvalue $i_1 + \cdots + i_p + j_1 + \cdots + j_q + m_1 + \cdots + m_r$.
The eigenvalue is called the weight of the vector in $V(k,0)$. We denote
the weight of $v$ by $\wt v$.

We consider two subalgebras $\wh = \C h \otimes \C[t,t^{-1}] \oplus \C C$
and $\wh_\ast = ( \oplus_{n \ne 0} \C h \otimes t^n) \oplus \C C$ of the
Lie algebra $\hsl$. Let $\Vh(k,0)$ be the subalgebra of
$V_{\hsl}(k,0)$ with basis $h(-i_1) \cdots h(-i_p)\1$, $i_1 \ge \cdots
\ge i_p \ge 1$. That is, $\Vh(k,0)$ is a vertex operator algebra
associated with the Heisenberg algebra $\wh_\ast$ of level $k$. The
conformal vector of $\Vh(k,0)$ is given by
\begin{equation}\label{eq:omega_gamma}
\omega_{\gamma} = \frac{1}{4k} h(-1)^2\1.
\end{equation}
Its central charge is $1$.

Now, $V(k,0)$ is completely reducible as a $V_{\wh}(k,0)$-module.
More precisely,
\begin{equation}\label{eq:dec-Heisenberg}
V(k,0) = \oplus_\lambda M_{\wh}(k,\lambda) \otimes N_\lambda.
\end{equation}
Here $M_{\wh}(k,\lambda)$ denotes an irreducible highest weight module
for $\wh$ with a highest weight vector $v_\lambda$ such that $h(0)v_\lambda
= \lambda v_\lambda$ and
\begin{equation*}
N_\lambda = \{ v \in V(k,0)\,|\, h(m)v = \lambda\delta_{m,0}v \text{ for
} m \ge 0\}.
\end{equation*}
The index $\lambda$ runs over all even integers, since the eigenvalues of
$h(0)$ in $V(k,0)$ are even integers. In fact, $h(0)$ acts as $2(q-r)$
on the vector of the form \eqref{eq:V-basis}.

In the case $\lambda = 0$, $M_{\wh}(k,0)$ is identical with
$V_{\wh}(k,0)$ and $N_0$ is the commutant \cite[Theorem 5.1]{FZ}
of $V_{\wh}(k,0)$ in $V(k,0)$.
The commutant $N_0$ is a vertex operator algebra with the conformal vector
$\omega = \omega_{\mraff} - \omega_{\gamma}$;
\begin{equation}\label{eq:omega}
\omega = \frac{1}{2k(k+2)} \Big( -kh(-2)\1 -h(-1)^2\1 + 2k e(-1)f(-1)\1
\Big),
\end{equation}
whose central charge is $3k/(k+2) - 1 = 2(k-1)/(k+2)$. Since the conformal
vector of $V_{\wh}(k,0)$ is $\omega_{\gamma}$, we have $N_0 = \{ v \in
V(k,0)\,|\, (\omega_\gamma)_0 v = 0\}$ by \cite[Theorem 5.2]{FZ}. It is
also the commutant of $\Vir(\omega_{\gamma})$ in $V(k,0)$, where
$\Vir(\omega_{\gamma})$ is the subalgebra of $V(k,0)$ generated by
$\omega_{\gamma}$. Since $\omega_1 v = (\omega_{\mraff})_1 v$ for $v \in
N_0$, the weight of $v$ in $N_0$ agrees with that in $V(k,0)$.

By a direct computation, we see that the dimension of the weight $i$
subspace $(N_0)_{(i)}$ of $N_0$ is $2$, $4$ and $6$ for $i=3$, $4$ and $5$,
respectively. Furthermore, we can verify that there is up to a scalar
multiple, a unique Virasoro primary vector $W^i$ in $(N_0)_{(i)}$ for $i =
3,4,5$. Here a Virasoro primary vector of weight $i$ means that $\omega_2
W^i = \omega_3 W^i = 0$ and $\omega_1 W^i  = i W^i$. In this paper, we take
\begin{equation}\label{eq:W3}
\begin{split}
W^3 &= k^2 h(-3)\1 + 3 k h(-2)h(-1)\1 +
2h(-1)^3\1 - 6k h(-1)e(-1)f(-1)\1 \\
& \quad + 3 k^2e(-2)f(-1)\1 - 3 k^2e(-1)f(-2)\1.
\end{split}
\end{equation}
As to $W^4$ and $W^5$, see Appendix \ref{App:def-W3-W4-W5}.

We denote by $\tW$ the subalgebra of $N_0$ generated by $\omega$, $W^3$,
$W^4$ and $W^5$. Actually, $\tW$ coincides with $W(2,3,4,5)$ of
\cite{BEHHH}.

\begin{rmk}
Our $W^3$, $W^4$ and $W^5$ are scalar multiples of $W_3$, $W_4$ and $W_5$
in the notation of \cite[Appendix A]{BEHHH}. In fact,
\begin{align*}
W^3 &= \frac{1}{2}W_3,\\
W^4 &= \frac{16k+17}{144k(2k+3)} W_4,\\
W^5 &= -\frac{64k+107}{3456k^2(2k+3)(3k+4)} W_5.
\end{align*}
Notice that $h$, $e$, $f$ and $\omega$ are denoted by $J^o$, $J^+$, $J^-$
and $L$, respectively in \cite{BEHHH}.
\end{rmk}

Recall that $W^i_n$ is a component operator, that is, the coefficient of
$x^{-n-1}$ in the vertex operator associated with $W^i$. The computation of
$W^i_n W^j$, $3 \le i \le j \le 5$, $0 \le n \le i+j-1$ has been done in
\cite{BEHHH}. In this paper, we compute $W^i_n W^j$ as an element of the
vertex operator algebra $V(k,0)$ by using the definition
\eqref{eq:def-vertex-operator} of the component operator $v_n$ and the
commutation relation \eqref{eq:affine-commutation} of the operators $a(n)$,
$a \in \{ h,e,f\}$, $n \in \Z$ on $V(k,0)$, together with the property
$a(n)\1 = 0$ for $n \ge 0$ of the vacuum vector $\1$. Every element of
$V(k,0)$ involved in the computation here is expressed as a linear
combination of the basis \eqref{eq:V-basis} of $V(k,0)$. The results can be
found in Appendix \ref{App:OPE} (see Remark \ref{Rem:tW-basis} also). For
instance,
\begin{equation}\label{eq:W3m1W3}
\begin{split}
W^3_1 W^3 &= -\big( 162 k^3 (k-2) (k+2) (3 k+4)/(16
k+17)\big)
\omega_{-3}\1 \\
& \quad + \big(288 k^3 (k-2) (k+2)^2 (3 k+4)/(16 k+17)\big)
\omega_{-1}\omega\\
& \quad + \big(36 k (2 k+3)/(16 k+17)\big) W^4,
\end{split}
\end{equation}
\begin{equation}\label{eq:W3m1W4}
\begin{split}
W^3_1 W^4 &= \big(1248 k^2 (k-3) (k+2) (2 k+1) (2 k+3)
/(64 k+107)\big) \omega_{-1}W^3\\
& \quad - \big(48 k^2 (k-3) (2 k+1) (2 k+3)
(2 k+7)/(64 k+107)\big) W^3_{-3}\1\\
& \quad - \big(12 k (3 k+4) (16 k+17) /(64 k+107)\big) W^5.
\end{split}
\end{equation}

Equation \eqref{eq:W3m1W3} can be obtained once we express $W^3_1 W^3$,
$\omega_{-3}\1$, $\omega_{-1}\omega$ and $W^4$ as linear combinations of
the basis \eqref{eq:V-basis}. The equations for the other $W^i_n W^j$'s are
obtained similarly. Each $W^i_n W^j$ can be described by using $\omega$,
$W^3$, $W^4$ and $W^5$. That is, $\tW$ is closed within these four
elements. We also notice that
\begin{equation*}
\begin{split}
W^3_5 W^3 &= 12 k^3 (k-2) (k-1) (3 k+4) \1,\\
W^3_4 W^3 &= 0,\\
W^3_3 W^3 &= 36k^3(k-2)(k+2)(3k+4)\omega.
\end{split}
\end{equation*}

\begin{rmk}\label{rmk:parameter-k}
Our computation of $W^i_nW^j$'s in $V(k,0)$ has been done by a computer
algebra system Risa/Asir. During the computation, we only use the condition
$a(n)\1 = 0$ for $a \in \{h,e,f\}$, $n \ge 0$, the commutation relation
\eqref{eq:affine-commutation} and the definition
\eqref{eq:def-vertex-operator} of vertex operators on $V(k,0)$. The
parameter $k$ is treated as a formal variable. That is, we do not assume
that $k$ is an integer greater than $1$ in the computation. Hence in
Appendix \ref{App:OPE}, $k$ can be considered as a formal variable.
\end{rmk}

\begin{rmk}\label{rmk:W3-generation}
From $W^3_3W^3$, $W^3_1W^3$ and $W^3_1W^4$, we see that $\tW$ is generated
by a single element $W^3$ if $k \ge 3$. It turns out that $W^3$, $W^4$ or
$W^5$ is contained in a maximal ideal of $\tW$ if $k = 2$, $3$ or $4$ (see
Section \ref{Sect:small-k-case} for detail). These are the degenerate
cases.
\end{rmk}

Let $\Vir(\omega)$ be the subalgebra of $\tW$ generated by $\omega$. Each
$W^i_n W^j$, $3 \le i \le j \le 5$, $0 \le n \le i+j-1$ is a linear
combination of elements in $\Vir(\omega)$, $\Vir(\omega)\cdot W^p_m\1$ with
$p < i+j$ and $m \le -1$, and $\Vir(\omega)\cdot W^r_k W^s_m\1$ with $r+s <
i+j$, $r \le s$ and $k, m \le -1$, where $p,r,s \in \{ 3,4,5\}$ (see
Appendix \ref{App:OPE}). Note also that $\omega_n W^i = 0$ if $n \ge 2$,
$\omega_1 W^i = i W^i$ and $\omega_0 W^i = W^i_{-2}\1$ for $i = 3,4,5$.
Hence, using basic formulas for a vertex operator algebra \cite[(3.1.9),
(3.1.12)]{LL}
\begin{align}
[u_m,v_n] &= \sum_{i \ge 0} \binom{m}{i} (u_iv)_{m+n-i},\label{eq:basic-formula1}\\
(u_m v)_n &= \sum_{i \ge 0} (-1)^i \binom{m}{i} \big(u_{m-i}v_{n+i} -
(-1)^m v_{m+n-i}u_i\big)\label{eq:basic-formula2},
\end{align}
we obtain the following lemma by induction.

\begin{lem}\label{spanning-set}
$\tW$ is spanned by the elements
\begin{equation}\label{eq:normal-form}
\omega_{-i_1} \cdots \omega_{-i_p} W^3_{-j_1} \cdots W^3_{-j_q} W^4_{-m_1}
\cdots W^4_{-m_r} W^5_{-n_1} \cdots W^5_{-n_s}\1
\end{equation}
with $i_1 \ge \cdots \ge i_p \ge 1$, $j_1 \ge \cdots \ge j_q \ge 1$, $m_1
\ge \cdots \ge m_r \ge 1$ and $n_1 \ge \cdots \ge n_s \ge 1$.
\end{lem}

An element of the form \eqref{eq:normal-form} is said to be of normal form.
Another notation is more convenient on some occasion. Set $L(n) =
\omega_{n+1}$, $W^3(n)=W^3_{n+2}$, $W^4(n)=W^4_{n+3}$ and
$W^5(n)=W^5_{n+4}$. All of these operators are of weight $-n$. The spanning
set of $\tW$ can also be described by
\begin{equation}\label{eq:normal-form2}
\begin{split}
L(-i_1) & \cdots L(-i_p) W^3(-j_1) \cdots W^3(-j_q)\\
& \qquad \cdot W^4(-m_1) \cdots W^4(-m_r) W^5(-n_1) \cdots W^5(-n_s)\1
\end{split}
\end{equation}
with $i_1 \ge \cdots \ge i_p \ge 2$, $j_1 \ge \cdots \ge j_q \ge 3$, $m_1
\ge \cdots \ge m_r \ge 4$ and $n_1 \ge \cdots \ge n_s \ge 5$. The weight of
the vector \eqref{eq:normal-form2} is
\begin{equation*}
i_1 + \cdots + i_p + j_1 + \cdots + j_q + m_1 + \cdots + m_r + n_1 + \cdots
+ n_s.
\end{equation*}

A vector $u$ of a $\tW$-module is called a highest weight vector for $\tW$
with highest weight $(a_2,a_3,a_4,a_5)$ if $L(n)u = W^i(n)u = 0$ for $n \ge
1$, $L(0)u = a_2 u$ and $W^i(0)u = a_i u$ for $i=3,4,5$. By a similar
argument as above, we see that the vectors
\begin{equation}\label{eq:normal-form3}
\begin{split}
L(-i_1) & \cdots L(-i_p) W^3(-j_1) \cdots W^3(-j_q)\\
& \qquad \cdot W^4(-m_1) \cdots W^4(-m_r) W^5(-n_1) \cdots W^5(-n_s)u
\end{split}
\end{equation}
with $i_1 \ge \cdots \ge i_p \ge 1$, $j_1 \ge \cdots \ge j_q \ge 1$, $m_1
\ge \cdots \ge m_r \ge 1$ and $n_1 \ge \cdots \ge n_s \ge 1$ span the
$\tW$-submodule $\tW\cdot u$ generated by such a highest weight vector $u$.

An automorphism of the Lie algebra $sl_2$ given by $h \mapsto -h$, $e
\mapsto f$, $f \mapsto e$ lifts to an automorphism $\theta$ of the vertex
operator algebra $V(k,0)$ of order $2$. The Virasoro element
$\omega_\gamma$ is invariant under $\theta$ by \eqref{eq:omega_gamma}.
Hence $\theta$ induces an automorphism of $N_0$. In fact, $\theta \omega =
\omega$ and $\theta W^3 = -W^3$ by \eqref{eq:omega} and \eqref{eq:W3},
respectively, and so $\theta W^4 = W^4$ and $\theta W^5 = -W^5$ by
\eqref{eq:W3m1W3} and \eqref{eq:W3m1W4}, respectively. More generally,
$\theta v = (-1)^{q+s}v$ for an element $v$ of the form
\eqref{eq:normal-form}. Let $\tW^{\pm} = \{ v \in \tW\,|\,\theta v = \pm
v\}$.

\begin{thm}\label{thm:aut-gp}
If $k \ge 3$, then the automorphism group $\Aut \tW$ of $\tW$ is $\la
\theta \ra$, a group of order $2$ generated by $\theta$.
\end{thm}

\begin{proof}
Let $\sigma \in \Aut \tW$. Since any Virasoro primary vector of weight $3$
in $N_0$ is a scalar multiple of $W^3$, we have $\sigma W^3 = \xi W^3$ for
some scalar $\xi \ne 0$. Then $\sigma(W^3_5W^3) = \xi^2 W^3_5W^3$. Now,
$W^3_5W^3$ is a nonzero scalar multiple of the vacuum vector $\1$ since we
are assuming that $k \ge 3$. Thus $\xi^2 = 1$ and the assertion holds, for
$\tW$ is generated by $W^3$.
\end{proof}

We will consider Zhu's algebra $A(\tW)$ of $\tW$. Zhu's algebra $A(V)$ of a
vertex operator algebra $(V, Y, \1, \omega)$ introduced by Zhu \cite{Zhu}
is very powerful for the study of irreducible $V$-modules. For $u,v \in V$
with $u$ being homogeneous, let
\begin{equation}\label{eq:ast-circ}
u \ast v = \sum_{i \ge 0} \binom{\wt u}{i} u_{i-1}v, \qquad u \circ v =
\sum_{i \ge 0} \binom{\wt u}{i} u_{i-2}v.
\end{equation}

We extend these two binary operations to arbitrary $u, v \in V$ by
linearity. The subspace $O(V)$ spanned by all $u \circ v$ with $u,v \in V$
is a two-sided ideal with respect to $\ast$. We denote by $[v]$ the image
of $v \in V$ in the quotient space $A(V) = V/O(V)$. Then $A(V)$ is an
associative algebra with the identity $[\1]$ by the binary operation $[u]
\ast [v] = [u \ast v]$. We write $u \sim v$ if $[u] = [v]$.

We need some formulas. By \cite[Lemma 2.1.2]{Zhu},
\begin{equation}\label{eq:circ-general}
\sum_{i \ge 0} \binom{\wt(u)+m}{i} u_{i-n-2}v \in O(V) \quad \text{for } n
\ge m \ge 0.
\end{equation}

By \eqref{eq:ast-circ}, $u_{-1}v$ can be written as a linear combination of
$u \ast v$ and $u_n v$, $n \ge 0$. Note that $\wt (u_n v) = \wt u + \wt v -
n -1$, $n \ge 0$ is strictly smaller than $\wt (u_{-1}v)$ for homogeneous
elements $u,v \in V$. For $n \ge 0$, \eqref{eq:circ-general} implies that
$[u_{-n-2}v]$ is a linear combination of $[u_{i-n-2}v]$, $i \ge 1$. As
before, note that $\wt (u_{i-n-2}v) < \wt(u_{-n-2}v)$ for $i \ge 1$.

Another useful formula \cite[Lemma 2.1.3]{Zhu} is
\begin{equation}\label{eq:Zhu-skew}
u \ast v - v \ast u  \sim \sum_{j \ge 0} \binom{\wt(u)-1}{j} u_j v.
\end{equation}

It follows from \cite[(4.2), (4.3)]{Wang} that
\begin{equation}\label{eq:Lnv}
[L(-n) v] = (-1)^n \big( (n-1)[\omega] \ast [v] + [L(0)v] \big) \quad
\text{for } n \ge 1.
\end{equation}
If $v$ is homogeneous, then $[L(-1)v] = -(\wt v)[v]$ and so
\cite[(1.2.17)]{Zhu} implies that
\begin{equation}\label{eq:Zhu-vm1}
[v_{-m-1}\1] = \frac{(-1)^m}{m !} \Big( \prod_{i=0}^{m-1} (\wt(v)+i) \Big)
[v]  \quad \text{for } m \ge 0.
\end{equation}

Let $o(v) = v_{\wt(v)-1}$ for a homogeneous element $v \in V$ and extend it
to an arbitrary element by linearity. If $U = \oplus_{n=0}^\infty U(n)$ is
an admissible $V$-module as in \cite{DLM} with $U(0) \ne 0$, then $o(v)$
acts on its top level $U(0)$. Zhu's theory \cite[Theorems 2.1.2 and
2.2.2]{Zhu} can be summarized as follows.

\medskip\noindent
(1) $o(u)o(v) = o(u \ast v)$ as operators on $U(0)$ and $o(v)$ acts as $0$
if $v \in O(V)$. Hence $U(0)$ is an $A(V)$-module, where $[v]$ acts as
$o(v)$ on $U(0)$.

\medskip\noindent
(2) $U \mapsto U(0)$ is a bijection between the set of equivalence classes
of irreducible admissible $V$-modules and the set of equivalence classes of
irreducible $A(V)$-modules.

\medskip
We now study Zhu's algebra $A(\tW) = \tW/O(\tW)$ of $\tW$. First, we show
that $A(\tW)$ is generated by $[\omega]$, $[W^3]$, $[W^4]$, and $[W^5]$.
For this purpose, take an element $v$ of the form \eqref{eq:normal-form}.
It is sufficient to confirm that $[v]$ is a linear combination of monomials
in $[\omega]$, $[W^3]$, $[W^4]$, and $[W^5]$. We proceed by induction on
weight. By \eqref{eq:Lnv}, we may assume that $v$ is
\begin{equation*}
W^3_{-j_1} \cdots W^3_{-j_q} W^4_{-m_1} \cdots W^4_{-m_r} W^5_{-n_1} \cdots
W^5_{-n_s}\1
\end{equation*}

Suppose $q \ge 1$. We separate $W^3_{-j_1}$ and the remaining part, say
$v'$, so that $v = W^3_{-j_1}v'$. If $j_1 \ge 2$, then $[W^3_{-j_1}v']$ is
a linear combination of the images in $A(\tW)$ of elements in $\tW$ having
strictly smaller weight than $v$ by \eqref{eq:circ-general}. If $j_1 = 1$,
then $[W^3_{-j_1}v']$ is a linear combination of $[W^3] \ast [v']$ and
$[W^3_n v']$, $n=0,1,2$. Since $\wt (W^3_n v') < \wt v$ for $n=0,1,2$, we
may assume that $q=0$ by induction on weight. We can apply a similar
argument to $W^4_{-m_1}$ or $W^5_{-n_1}$ if $r \ge 1$ or $s \ge 1$. In this
way, we see that $[v]$ is a linear combination of monomials in $[\omega]$,
$[W^3]$, $[W^4]$, and $[W^5]$ as required. Actually, \eqref{eq:Zhu-vm1} is
useful in computing $[v]$.

Next, we show that $A(\tW)$ is commutative. By a property of Zhu's algebra,
$[\omega]$ lies in the center of $A(\tW)$ \cite[Theorem 2.1.1]{Zhu}. It
follows from \eqref{eq:Zhu-skew} that
\begin{equation*}
W^3 \ast W^4 - W^4 \ast W^3 = W^3_0 W^4 + 2W^3_1 W^4 + W^3_2 W^4.
\end{equation*}

Using an explicit expression of $W^3_n W^4$, $n = 0,1,2$ as a linear
combination of vectors of normal form \eqref{eq:normal-form} given in
Appendix \ref{App:OPE}, we can describe its image $[W^3_n W^4]$ in $A(\tW)$
as a polynomial in $[\omega]$, $[W^3]$ and $[W^5]$. In fact, we can verify
that
\begin{equation*}
[W^3] \ast [W^4] - [W^4] \ast [W^3] = [W^3_0 W^4] + 2[W^3_1 W^4] + [W^3_2
W^4] = 0.
\end{equation*}

Likewise, we have
\begin{align*}
[W^3] \ast [W^5] - [W^5] \ast [W^3] &= [W^3_0 W^5] + 2[W^3_1 W^5] + [W^3_2
W^5]
= 0,\\
[W^4] \ast [W^5] - [W^5] \ast [W^4] &= [W^4_0 W^5] + 3[W^4_1 W^5] + 3[W^4_2
W^5] + [W^4_3 W^5] = 0.
\end{align*}

Thus $A(\tW)$ is commutative. We have obtained the following lemma.
\begin{lem}\label{lem:Zhu-tW}
$A(\tW)$ is commutative and it is generated by $[\omega]$, $[W^3]$, $[W^4]$
and $[W^5]$.
\end{lem}

The above lemma implies that $w_2 \mapsto [\omega]$, $w_3 \mapsto [W^3]$,
$w_4 \mapsto [W^4]$, $w_5 \mapsto [W^5]$ define a homomorphism $\tvp$ of
associative algebras from a polynomial algebra $\C[w_2,w_3,w_4,w_5]$ of
four variables $w_2, w_3, w_4, w_5$ onto $A(\tW)$. In particular, $A(\tW)$
is spanned by
\begin{equation}\label{eq:Zhu-element}
[\om]^{\ast p} \ast [W^3]^{\ast q} \ast [W^4]^{\ast r} \ast [W^5]^{\ast s},
\quad p,q,r,s \ge 0,
\end{equation}
where $[u]^{\ast p}$ is a product of $p$ copies of $[u]$ in $A(\tW)$.

We will study linear relations among vectors of normal form
\eqref{eq:normal-form} of small weight. The generating function of the
number of vectors of normal form with respect to weight is
\begin{equation*}
\frac{(1-q)^4(1-q^2)^3(1-q^3)^2(1-q^4)}{\prod_{n \ge 1}(1-q^n)^4}.
\end{equation*}
The sum of its first several terms are
\begin{equation*}
1 + q^2 + 2q^3 + 4q^4 + 6q^5 + 11q^6 + 16q^7 + 29q^8 + 44q^9 + 72q^{10} +
\cdots.
\end{equation*}

We express all vectors of normal form \eqref{eq:normal-form} of weight at
most $10$ as linear combinations of the basis \eqref{eq:V-basis} of
$V(k,0)$. By a direct calculation, we can verify that those vectors of
normal form of weight at most $7$ are all linearly independent. However,
this is not the case if the weight is greater than $7$
\cite[(2.1.9)]{BEHHH}. There are $29$ vectors of normal form of weight $8$,
which span a subspace of dimension $27$. Thus the dimension of the weight
$8$ subspace of $\tW$ is $27$. If we eliminate $(W^3_{-2})^2\1$ and
$W^3_{-1}W^4_{-2}\1$, then the remaining $27$ vectors form a basis of the
weight $8$ subspace of $\tW$. That is, there are two nontrivial linear
relations in the weight $8$ subspace, one involves $(W^3_{-2})^2\1$ and the
other involves $W^3_{-1}W^4_{-2}\1$. Such a nontrivial linear relation is
called a null field \cite{BEHHH, Hornfeck}. Recall that $\tW^{\pm} = \{ v
\in \tW\,|\,\theta v = \pm v\}$. Among $29$ vectors of normal form
\eqref{eq:normal-form} of weight $8$, we see that $17$ are contained in
$\tW^+$ and the remaining $12$ are contained in $\tW^-$. Note that
$(W^3_{-2})^2\1 \in \tW^+$ and $W^3_{-1}W^4_{-2}\1 \in \tW^-$. Thus there
is up to a scalar multiple, a unique nontrivial linear relation in each
weight $8$ subspace of $\tW^+$ and $\tW^-$. Let $\bv^0$ be a nontrivial
linear relation in the weight $8$ subspace of $\tW^+$. It is a linear
combination of the $17$ vectors of normal form of weight $8$ in $\tW^+$. An
explicit form of such a linear combination can be found in Appendix
\ref{App:wt8-linear-relations}. It is obtained by describing each of those
$17$ vectors as a linear combination of the basis \eqref{eq:V-basis}.
Actually, $\bv^0 = 0$ in $V(k,0)$.

We can express the image of each of those $17$ vectors of normal form in
Zhu's algebra $A(\tW)$ as a linear combination of the elements of the form
\eqref{eq:Zhu-element} by using a similar argument as in the proof of Lemma
\ref{lem:Zhu-tW}. Then the image $[\bv^0]$ of $\bv^0$ in $A(\tW)$ becomes a
linear combination of the elements of the form \eqref{eq:Zhu-element}.
Replace $[\omega]$, $[W^3]$, $[W^4]$ and $[W^5]$ with $w_2$, $w_3$, $w_4$
and $w_5$, respectively in $[\bv^0]$. Then we obtain
\begin{align*}
Q_0 =& -8 k^4 (k+2)^2 (3 k+4) (4 k-1) (64 k+107) (k^2+k+1) w_2^2\\
& \quad + 4 k^4 (k+2)^3 (3 k+4) (64 k+107) (80 k^2+30 k+61) w_2^3\\
& \quad -112 k^4 (k+2)^4 (3 k+4) (6 k-5) (64 k+107) w_2^4\\
& \quad + 2 k (16 k+17)^2 (k^2+3 k+5) w_3^2\\
& \quad + k (k+2) (16 k+17)^2 (26 k+83) w_2 w_3^2\\
& \quad + 2 k^2 (k+2) (64 k+107) (8 k^2+9 k-8) w_2 w_4\\
& \quad -4 k^2 (k+2)^2 (36 k+61) (64 k+107) w_2^2 w_4\\
& \quad + 2 (64 k+107) w_4^2 + (16 k+17)^2 w_3 w_5.
\end{align*}

Since $\bv^0 = 0$, the above polynomial lies in the kernel of $\tvp$. One
may discuss the image in Zhu's algebra of a nontrivial linear relation in
the weight $8$ subspace of $\tW^-$. However, the polynomial obtained in
this manner becomes $0$. Thus the null field of weight $8$ in $\tW^-$ gives
no information on $A(\tW)$.

Next, we study null fields of weight $9$. There are $44$ vectors of normal
form \eqref{eq:normal-form} of weight $9$. Among them, $22$ vectors are
contained in $\tW^+$ and the other 22 vectors are contained in $\tW^-$. We
eliminate $W^3_{-3}W^3_{-2}\1$, $W^3_{-2}W^4_{-2}\1$, $W^3_{-1}W^4_{-3}\1$
and $W^3_{-1}W^5_{-2}\1$. Then the remaining $40$ vectors form a basis of
the weight $9$ subspace of $\tW$. We take a nontrivial linear relation
$\bv^1$ in $\tW^-$ which involves $W^3_{-1}W^4_{-3}\1$ (see Appendix
\ref{App:wt9-linear-relations}). We calculate $[\bv^1]$ by a similar method
as above and obtain the following polynomial.
\begin{align*}
Q_1 =& -16 k^3 (k+2) (2 k+1) (13 k^3+24 k^2+7 k+10) w_2 w_3\\
& \quad + 4 k^3 (k+2)^2 (1040 k^3+2232 k^2+1213 k+1116) w_2^2 w_3\\
& \quad -16 k^3 (k+2)^3 (674 k^2+637 k-1100) w_2^3 w_3\\
& \quad + (16 k+17) (64 k+107) w_3^3
 + 2 k (68 k^2+119 k+20) w_3 w_4\\
& \quad -4 k (k+2) (358 k+559) w_2 w_3 w_4
 + 4 k^2 (k+2) (3 k+4) (4 k-1) w_2 w_5\\
& \quad -112 k^2 (k+2)^2 (3 k+4) w_2^2 w_5 + 4 w_4 w_5.
\end{align*}

As before, $\bv^1 = 0$ in $V(k,0)$ and the polynomial $Q_1$ lies in the
kernel of $\tvp$. No further relation in $A(\tW)$ is obtained from the null
fields of weight $9$.

Let $C_2(\tW)$ be the subspace of $\tW$ spanned by the elements $u_{-2}v$
with $u,v \in \tW$. The quotient space $\tW/C_2(\tW)$ has a commutative
Poisson algebra structure \cite[Section 4.4]{Zhu}. Since $(L(-1)v)_n =
-nv_{n-1}$, we have $u_{-m}v \in C_2(\tW)$ for $m \ge 2$. By Lemma
\ref{spanning-set}, we see that $x_2 \mapsto \omega + C_2(\tW)$, $x_i
\mapsto W^i + C_2(\tW)$, $i=3,4,5$ define a homomorphism $\tilde{\rho}$ of
associative algebras from a polynomial algebra $\C[x_2,x_3,x_4,x_5]$ of
four variables $x_2,x_3,x_4,x_5$ onto $\tW/C_2(\tW)$.

We consider the images of some null fields in $\tW/C_2(\tW)$. The
nontrivial linear relation involving $W^3_{-1}W^4_{-2}\1$ is written by
vectors of normal form contained in $C_2(\tW)$ (see Appendix
\ref{App:wt8-linear-relations}). Hence its image in $\tW/C_2(\tW)$ is
trivial. On the other hand, $(W^3_{-2})^2\1$ is a linear combination of
vectors of normal form not all of which lie in $C_2(\tW)$ (see Appendix
\ref{App:wt8-linear-relations}). Therefore, its image in $\tW/C_2(\tW)$ is
a nonzero polynomial in $\omega + C_2(\tW)$, $W^i + C_2(\tW)$, $i=3,4,5$.
Replace $\omega + C_2(\tW)$ and $W^i + C_2(\tW)$ with $x_2$ and $x_i$,
$i=3,4,5$, respectively in the polynomial and multiply it by
\begin{equation*}
(17/9)k(k+1)(16k+17)^2(64k+107).
\end{equation*}
Let $B_0$ be the polynomial obtained in this manner. Then
\begin{align*}
B_0 &= -112 k^4 (k+2)^4 (3 k+4) (6 k-5) (64 k+107) x_2^4\\
& \quad +k (k+2) (16 k+17)^2 (26 k+83) x_2 x_3^2\\
& \quad -4 k^2 (k+2)^2 (36 k+61) (64 k+107) x_2^2 x_4\\
& \quad + 2 (64 k+107) x_4^2 + (16 k+17)^2 x_3 x_5.
\end{align*}

Since $(W^3_{-2})^2\1 \in C_2(\tW)$, $B_0$ lies in the kernel of
$\tilde{\rho}$. Likewise, we consider the images of the four nontrivial
linear relations among the vectors of normal form of weight $9$ in
$\tW/C_2(\tW)$. We obtain only one nonzero polynomial $B_1$ up to a scalar
multiple in this manner, namely,
\begin{align*}
B_1 &= 16 k^3 (k+2)^3 (674 k^2+637 k-1100) x_2^3 x_3
-(16 k+17) (64 k+107) x_3^3 \\
& \quad + 4 k (k+2) (358 k+559) x_2 x_3 x_4
+ 112 k^2 (k+2)^2 (3 k+4) x_2^2 x_5 -4 x_4 x_5.
\end{align*}
In fact, the polynomial $B_1$ comes from the null field $\bv^1$.

The weight $10$ subspace of $\tW^+$ is of dimension $35$, while there are
$40$ vectors of normal form in the subspace. We eliminate
$\omega_{-1}(W^3_{-2})^2\1$, $(W^3_{-3})^2\1$, $(W^4_{-2})^2\1$,
$W^3_{-2}W^5_{-2}\1$ and $W^3_{-1}W^5_{-3}\1$. Then the remaining $35$
vectors form a basis of the weight $10$ subspace of $\tW^+$. For instance,
$(W^3_{-3})^2\1$ can be written uniquely as a linear combination of those
$35$ vectors just as in the case for weight $8$ or $9$. We denote by
$\bv^2$ the nontrivial linear relation in $\tW^+$ involving
$(W^3_{-3})^2\1$.\footnote{We omit the explicit form of $\bv^2$ in Appendix
\ref{App:linear-relations}, for it is very complicated.} Take the image of
$\bv^2$ in $\tW/C_2(\tW)$ and replace $\omega + C_2(\tW)$ and $W^i +
C_2(\tW)$ with $x_2$ and $x_i$, $i=3,4,5$, respectively. Then we obtain a
polynomial in $x_2,x_3,x_4,x_5$. Its suitable constant multiple $B_2$ is as
follows.
\begin{align*}
B_2 &= 16 k^5 (k+2)^5 (6 k-5) (64 k+107)\\
& \qquad \quad \cdot (2734552 k^4+10654776 k^3+10502887 k^2-4070244 k-7670000) x_2^5\\
& \quad -k^2 (k+2)^2 (16 k+17) (29745920 k^5+282149936 k^4\\
& \qquad \quad +715730704 k^3+459700602 k^2-375262083 k-379918040) x_2^2 x_3^2\\
& \quad -20 k^3 (k+2)^3 (64 k+107)\\
& \qquad \quad \cdot (81056 k^4-691736 k^3-2503316 k^2-1005811 k+1451208) x_2^3 x_4\\
& \quad + 17 (k+1) (16 k+17) (64 k+107)^3 x_3^2 x_4\\
& \quad -2 k (k+2) (64 k+107) (979064 k^3+3791032 k^2+4574059 k+1616792) x_2 x_4^2\\
& \quad -k (k+2) (16 k+17)^2 (256632 k^3+825008 k^2+598779 k-114896) x_2 x_3 x_5\\
& \quad -34 (k+1) (16 k+17)^2 (64 k+107) x_5^2.
\end{align*}

We note that $B_0$, $B_1$ and $B_2$ lie in the kernel of $\tilde{\rho}$.
These polynomials will be used in Section \ref{Sect:small-k-case}.

\begin{rmk}\label{Rem:tW-basis}
We eliminate $(W^3_{-2})^2\1$ and $W^3_{-1}W^4_{-2}\1$ from the vectors of
normal form \eqref{eq:normal-form} of weight $8$ and $W^3_{-3}W^3_{-2}\1$,
$W^3_{-2}W^4_{-2}\1$, $W^3_{-1}W^4_{-3}\1$ and $W^3_{-1}W^5_{-2}\1$ from
the vectors of normal form of weight $9$ when we express $W^i_n W^j$, $3
\le i \le j \le 5$, $0 \le n \le i+j-1$ as linear combinations of vectors
of normal form in Appendix \ref{App:OPE}.
\end{rmk}

\section{Unique maximal ideal $\tI$ of $N_0$}\label{Sect:maximal-ideal-tI}
The vertex operator algebra $V(k,0)$ is not simple, for we are assuming
that $k$ is an integer greater than $1$. In fact, it possesses a unique
maximal ideal $\J$, which is generated by a weight $k+1$ vector
$e(-1)^{k+1}\1$ \cite{K}. The quotient algebra $L(k,0) = V(k,0)/\J$ is the
simple vertex operator algebra associated with an affine Lie algebra
$\widehat{sl}_2$ of type $A_1^{(1)}$ with level $k$. Since the Heisenberg
vertex operator algebra $V_{\wh}(k,0)$ is a simple subalgebra of $V(k,0)$,
it follows that $\J \cap V_{\wh}(k,0) = 0$. Thus $V_{\wh}(k,0)$ can be
considered as a subalgebra of $L(k,0)$. Just as in
\eqref{eq:dec-Heisenberg}, we have that $L(k,0)$ is a completely reducible
$V_{\wh}(k,0)$-module and we can write
\begin{equation}
L(k,0) = \oplus_\lambda M_{\wh}(k,\lambda) \otimes K_\lambda,
\end{equation}
where
\begin{equation*}
K_\lambda = \{v \in L(k,0)\,|\, h(m)v = \lambda\delta_{m,0}v \text{ for } m
\ge 0\}.
\end{equation*}

Note that $K_0$ is the commutant of $V_{\wh}(k,0)$ in $L(k,0)$. Similarly,
$\J$ is completely reducible as a $V_{\wh}(k,0)$-module. Hence by
\eqref{eq:dec-Heisenberg},
\begin{equation*}
\J = \oplus_\lambda M_{\wh}(k,\lambda) \otimes (\J \cap N_\lambda).
\end{equation*}
In particular, $\tI = \J \cap N_0$ is an ideal of $N_0$ and $K_0 \cong
N_0/\tI$.

\begin{lem}\label{lem:unique-max-ideal}
$\tI$ is a unique maximal ideal of $N_0$.
\end{lem}

\begin{proof}
The top level of the subalgebra $V_{\wh}(k,0) \otimes N_0$ of $V(k,0)$ is
$\C\1$, and so it is the top level of $N_0$ also. Hence $N_0$ has a unique
maximal ideal, say $S$. Since $\J$ does not contain $\1$, we have $\tI
\subset S$. The subspace $V(k,0)\cdot S$ spanned by $u_n v$, $u \in
V(k,0)$, $v \in S$, $n \in \Z$ is an ideal of $V(k,0)$. Let $u \in
M_{\wh}(k,\lambda) \otimes N_\lambda$ and $v \in S$. Since $S \subset
V_{\wh}(k,0) \otimes N_0$ and $M_{\wh}(k,\lambda) \otimes N_\lambda$ is a
$V_{\wh}(k,0) \otimes N_0$-module, the skew symmetry \cite[(3.1.30)]{LL} in
the vertex operator algebra $V(k,0)$ implies that $u_n v$ lies in
$M_{\wh}(k,\lambda) \otimes N_\lambda$. Then we see from
\eqref{eq:dec-Heisenberg} that the intersection of $V(k,0)\cdot S$ and
$V_{\wh}(k,0) \otimes N_0$ is $V_{\wh}(k,0) \otimes S$. Thus $V(k,0)\cdot
S$ is a proper ideal of $V(k,0)$ and so it is contained in the unique
maximal ideal $\J$. Therefore, $\tI = S$ as required.
\end{proof}

The vector $e(-1)^{k+1}\1$ is not contained in $N_0$. For $r \ge 1$ and an
integer $n$, we calculate that
\begin{align*}
h(n)f(0)^r e(-1)^r\1 &= \big( -2r f(0)^{r-1}f(n) + f(0)^rh(n)\big) e(-1)^r\1\\
&= -2r f(0)^{r-1}f(n)e(-1)^r\1 + 2r f(0)^r e(-1)^{r-1} e(n-1)\1\\
& \quad + f(0)^re(-1)^r h(n)\1
\end{align*}
in $V(k,0)$ by using \eqref{eq:affine-commutation}. Moreover,
\begin{equation*}
f(n)e(-1)^r\1 =
\begin{cases}
r(k-r+1)e(-1)^{r-1}\1 & \text{if $n=1$},\\
0 & \text{if $n \ge 2$}.
\end{cases}
\end{equation*}
Hence, $h(n)f(0)^r e(-1)^r\1 = 0$ if $n = 0$ or $n \ge 2$, and
\begin{align*}
h(1)f(0)^r e(-1)^r\1 &= -2r f(0)^{r-1}f(1)e(-1)^r\1\\
&= -2r^2(k - r +1)f(0)^{r-1}e(-1)^{r-1}\1.
\end{align*}

In particular, $h(n)f(0)^{k+1} e(-1)^{k+1}\1 = 0$ for $n \ge 0$. This
proves the next theorem.
\begin{thm}\label{thm:singular-vec}
$f(0)^{k+1}e(-1)^{k+1}\1 \in \tI$.
\end{thm}

It seems natural to expect the following properties of $f(0)^{k+1}
e(-1)^{k+1}\1$ (see Section \ref{Sect:small-k-case}).

\begin{conjecture}\label{Conj:ideal-generator}
$(1)$ The unique maximal ideal $\tI$ of $N_0$ is generated by a weight
$k+1$ vector $f(0)^{k+1}e(-1)^{k+1}\1$.

$(2)$ The automorphism $\theta$ acts as $(-1)^{k+1}$ on
$f(0)^{k+1}e(-1)^{k+1}\1$.
\end{conjecture}

We have considered the elements $\omega_{\mraff}$, $\omega_\gamma$,
$\omega$, $W^3$, $W^4$ and $W^5$ of $V(k,0)$. For simplicity of notation,
we use the same symbols to denote their images in $L(k,0) = V(k,0)/\J$.
Then $\omega$ is the conformal vector of $K_0$. Moreover, $K_0 = \{ v \in
L(k,0)\,|\, (\omega_\gamma)_0 v = 0\}$ by \cite[Theorem 5.2]{FZ}. It is
also the commutant of $\Vir(\omega_\gamma)$ in $L(k,0)$. The automorphism
$\theta$ of $V(k,0)$ induces an automorphism of $L(k,0)$. We denote it by
the same symbol $\theta$. Let $\W$ be a subalgebra of $K_0$ generated by
$\omega$, $W^3$, $W^4$ and $W^5$. Thus $\W$ is a homomorphic image of
$\tW$. We are interested in $\W$. The following theorem is a direct
consequence of Remark \ref{rmk:W3-generation} and Theorem
\ref{thm:aut-gp}.

\begin{thm}\label{thm:two-properties-W}
If $k \ge 3$, then the vertex operator algebra $\W$ is generated by $W^3$
and its automorphism group $\Aut \W = \langle \theta \rangle$ is of order
$2$.
\end{thm}

\section{Irreducible modules for $K_0$}\label{Sect:VOA-W}
In the preceding section, $K_0$ is defined to be the commutant of
$V_{\wh}(k,0)$ in $L(k,0)$. We will follow the argument in \cite{DL} to
realize $K_0$ in a vertex operator algebra associated with a lattice and
construct some of its irreducible modules. Those irreducible $K_0$-modules
will be denoted by $M^{i,j}$, $0 \le i \le k$, $0 \le j \le k-1$. Actually,
$M^{0,0} = K_0$.

Let $L = \Z\al_1 + \cdots + \Z\al_k$ with $\la \al_p,\al_q\ra = 2\dl_{pq}$.
Thus $L$ is an orthogonal sum of $k$ copies of a root lattice of type
$A_1$. Its dual lattice is $L^\perp = \frac{1}{2}L$. The commutator map
$c_0(\,\cdot\, ,\,\cdot\,)$ defined in \cite[Remark 6.4.12]{LL} is trivial
on $L^\perp$ so that the twisted group algebras $\C\{L\}$ and
$\C\{L^\perp\}$ considered in \cite[Section 6.4]{LL} are isomorphic to
ordinary group algebras $\C[L]$ and $\C[L^\perp]$ of additive groups $L$
and $L^\perp$, respectively. Denote a standard basis of $\C[L^\perp]$ by
$e^\alpha$, $\alpha \in L^\perp$ with multiplication $e^\alpha e^\beta =
e^{\alpha+\beta}$. We consider the vertex operator algebra $V_L = M(1)
\otimes \C[L]$ associated with the lattice $L$ and its module $V_{L^\perp}
= M(1) \otimes \C[L^\perp]$. Here, $M(1) = [\alpha_p(-n); 1 \le p \le k, n
\ge 1]$ as a vector space. The vertex operator algebra $V_L$ is spanned by
$\be_1(-n_1)\cdots\be_r(-n_r)e^\al$, $\be_i \in \{\al_1,\ldots,\al_k\}$,
$\al \in L$, $n_1 \ge \cdots \ge n_r \ge 1$. Let $\gm = \al_1 + \cdots +
\al_k$. Thus $\la \gm,\gm \ra = 2k$. Set
\begin{equation*}
H = \gamma(-1)\1, \quad E = e^{\al_1} + \cdots + e^{\al_k}, \quad F =
e^{-\al_1} + \cdots + e^{-\al_k},
\end{equation*}
where $\1 = 1 \otimes 1$ is the vacuum vector.

Then $H_0 H = 0$, $H_1 H = 2k\1$, $H_0 E = 2E$, $H_1 E = 0$, $H_0 F = -2F$,
$H_1 F = 0$, $E_0 F = H$, $E_1 F = k\1$, $E_0 E = E_1 E = F_0 F = F_1 F =
0$ and $A_n B = 0$ for $A, B \in \{ H, E, F\}$, $n \ge 2$ in the vertex
operator algebra $V_L$. Therefore, the component operators $H_n$, $E_n$ and
$F_n$, $n \in \Z$, which are endomorphisms of $V_L$ or $V_{L^\perp}$, give
a representation of $\widehat{sl}_2$ with level $k$ under the
correspondence
\begin{equation}\label{eq:HhEeFf}
H_n \leftrightarrow h(n),\qquad E_n  \leftrightarrow e(n),\qquad F_n
\leftrightarrow f(n).
\end{equation}

We have
\begin{equation*}
(E_{-1})^j\1 = j\,! \sum_{\substack{I \subset \{1,2,\ldots,k\}\\|I|=j}}
e^{\al_I}, \qquad (F_{-1})^j\1 = j\,! \sum_{\substack{I \subset
\{1,2,\ldots,k\}\\|I|=j}} e^{-\al_I},
\end{equation*}
and $(E_{-1})^{k+1}\1 = (F_{-1})^{k+1}\1 = 0$, where $\al_I = \sum_{p \in
I} \al_p$ for a subset $I$ of $\{1,2,\ldots,k\}$. In particular,
$(E_{-1})^k \1 = k \,! e^\gm$ and $(F_{-1})^k \1 = k \,! e^{-\gm}$.
Moreover,
\begin{equation*}
e^{-\gm}\cdot (E_{-1})^j\1 = \frac{j\,!}{(k-j)\,!} (F_{-1})^{k-j}\1,
\qquad e^{\gm}\cdot (F_{-1})^j\1 = \frac{j\,!}{(k-j)\,!}
(E_{-1})^{k-j}\1
\end{equation*}
in the group algebra $\C[L]$.

Let $\Vaff$ (resp. $V^\gamma$) be the subalgebra of $V_L$ generated by $H$,
$E$ and $F$ (resp. $e^\gm$ and $e^{-\gm}$). Then $V_L \supset \Vaff \supset
V^\gm$ with $\Vaff \cong L(k,0)$ and $V^\gm \cong V_{\Z\gm}$, where
$V_{\Z\gm}$ is the vertex operator algebra associated with a rank one
lattice $\Z\gamma$. Note that $V^\gm$ contains $(e^\gm)_{2k-2}e^{-\gm} =
H$. The $-1$ isometry of the lattice $L^\perp$ lifts to a linear
isomorphism $\theta$ of $V_{L^\perp}$ onto itself of order $2$. Such a lift
is unique, since the commutator map $c_0(\,\cdot\, ,\,\cdot\,)$ is trivial.
Its restriction to $V_L$ is an automorphism of the vertex operator algebra
$V_L$. We have $\theta H = -H$, $\theta E = F$ and $\theta F = E$. Thus the
restriction of the automorphism $\theta$ here to $\Vaff$ agrees with the
automorphism of $L(k,0) = V(k,0)/\J$ induced by the automorphism $\theta$
of $V(k,0)$ discussed in Section 2.

For simplicity of notation, we use the same symbols $\omega_{\mraff}$,
$\omega_\gamma$, $\omega$, $W^3$, $W^4$ and $W^5$ to denote their images in
$\Vaff \cong L(k,0)$ under the correspondence \eqref{eq:HhEeFf}. Then
\begin{align*}
\omega_{\mraff} &= \frac{1}{2(k+2)} \Big( \frac{1}{2}H_{-1}H +
E_{-1}F + F_{-1}E \Big),\\
\omega_{\gamma} &= \frac{1}{4k} H_{-1}H,\\
\omega &= \frac{1}{2k(k+2)} \Big( -H_{-1}H + k(E_{-1}F + F_{-1}E)
\Big),
\end{align*}
\begin{equation*}
\begin{split}
W^3 &= k^2 H_{-3}\1 + 3 k H_{-2}H +
2H_{-1}H_{-1}H - 6k H_{-1}E_{-1}F \\
& \quad + 3 k^2(E_{-2}F - E_{-1}F_{-2}\1).
\end{split}
\end{equation*}

Let $L' = \oplus_{p=1}^{k-1} \Z(\al_p - \al_{p+1})$, which is $\sqrt{2}$
times a root lattice of type $A_{k-1}$. We have
\begin{equation}\label{eq:L-prime}
L' = \{ \alpha \in L\,|\, \la \alpha,\gamma\ra = 0 \}.
\end{equation}
Moreover,
\begin{equation}\label{eq:L'-cosets-L}
L = \bigcup_{i=0}^{k-1} \, (i\al_1 + L' \oplus \Z\gamma)
\end{equation}
is a coset decomposition of $L$ by $L' \oplus \Z\gamma$ \cite[(4.1)]{LY}.

Let $V_{L'} = M(1)' \otimes \C[L']$ be the vertex operator algebra
associated with $L'$, where $M(1)' = [(\alpha_p-\alpha_q)(-n); 1 \le p < q
\le k, n \ge 1]$ as a vector space and $\C[L']$ is an ordinary group
algebra of the additive group $L'$. Then $V_{L'} \cong V_{\sqrt{2}A_{k-1}}$
is a subalgebra of $V_L$. By \eqref{eq:L-prime} and \eqref{eq:L'-cosets-L},
the following proposition holds \cite{LY} (see \cite[Theorem 5.2]{FZ}
also).

\begin{prop}\label{prop:VL-prime}
We have $V_{L'} = \{ v \in V_L \,|\, (\omega_\gamma)_n v = 0 \text{ for } n
\ge 0\}$. In particular, $K_0 \cong \Vaff \cap V_{L'}$.
\end{prop}

We will describe $\omega$ and $W^3$ explicitly as elements of $V_{L'}$ (see
\cite[Lemma 4.1]{LY}). In the vertex operator algebra $V_L$ we have
\begin{equation*}
H_{-1}H = \sum_{1 \le p \le k} \al_p(-1)^2\1 + \sum_{\substack{1 \le p, q
\le k\\ p \ne q}} \al_p(-1)\al_q(-1)\1,
\end{equation*}
\begin{equation*}
E_{-1}F + F_{-1}E = \sum_{1 \le p \le k} \al_p(-1)^2\1 + 2
\sum_{\substack{1 \le p, q \le k\\ p \ne q}} e^{\al_p-\al_q}.
\end{equation*}

Now, $(\al_p - \al_q)(-1)^2 = \al_p(-1)^2 + \al_q(-1)^2 - 2
\al_p(-1)\al_q(-1)$ and so
\begin{equation*}
\frac{1}{2} \sum_{\substack{1 \le p, q \le k\\ p \ne q}} (\al_p -
\al_q)(-1)^2 = (k - 1) \sum_{1 \le p \le k} \al_p(-1)^2 -
\sum_{\substack{1 \le p, q \le k\\ p \ne q}} \al_p(-1)\al_q(-1).
\end{equation*}
Hence we obtain (see \cite[Lemma 4.1]{LY})
\begin{equation}\label{eq:omega-in-VL'}
\omega = \frac{1}{4k(k+2)} \sum_{\substack{1 \le p, q \le k\\ p \ne q}}
(\al_p - \al_q)(-1)^2\1 + \frac{1}{k+2} \sum_{\substack{1 \le p, q \le k\\
p \ne q}} e^{\al_p-\al_q}.
\end{equation}

Moreover, we calculate that
\begin{equation*}
H_{-3}\1 = \sum_{1 \le p \le k} \al_p(-3)\1,
\end{equation*}
\begin{equation*}
H_{-2}H = \sum_{1 \le p, q \le k} \al_p(-2)\al_q(-1)\1,
\end{equation*}
\begin{equation*}
H_{-1}H_{-1}H = \sum_{1 \le p, q, r \le k} \al_p(-1)\al_q(-1)\al_r(-1)\1,
\end{equation*}
\begin{equation*}
H_{-1}E_{-1}F = \Big( \sum_{1 \le p \le k} \al_p(-1) \Big) \Big( \sum_{1
\le q \le k} \frac{1}{2}\big( \al_q(-2)\1 + \al_q(-1)^2\1 \big) +
\sum_{\substack{1 \le q, r \le k\\ q \ne r}} e^{\al_q-\al_r} \Big),
\end{equation*}
\begin{equation*}
\begin{split}
E_{-2}F - E_{-1}F_{-2}\1 &= \frac{1}{3} \sum_{1 \le p \le k} (
-\al_p(-3)\1 + \al_p(-1)^3\1)\\
& \quad + \sum_{\substack{1 \le p, q \le k\\ p \ne q}} \big(\al_p(-1) +
\al_q(-1)\big) e^{\al_p-\al_q}.
\end{split}
\end{equation*}

Using these equations, we have
\begin{equation}\label{eq:W3-in-VL'}
\begin{split}
W^3 &= \sum_{1 \le p, q, r \le k} (\al_p - \al_q)(-1)^2 (\al_p -
\al_r)(-1)\1\\
& \quad - 3k \sum_{\substack{1 \le q, r \le k\\ q \ne r}} \Big(
\sum_{\substack{1 \le p \le k\\ p \ne q}} (\al_p - \al_q)(-1) +
\sum_{\substack{1 \le p \le k\\ p \ne r}} (\al_p - \al_r)(-1) \Big)
e^{\al_q-\al_r}.
\end{split}
\end{equation}

Equation \eqref{eq:W3-in-VL'} in particular implies that $W^3 = 0$ in
$V_{L'}$ if $k = 2$ (see Section \ref{subsec:k-is-2}). We identify $\Vaff$
with $L(k,0)$ and $V^\gamma$ with $V_{\Z\gamma}$ from now on.

It is well known that the vertex operator algebra associated with a
positive definite even lattice is rational \cite{Dong}. We study a
decomposition of $\Vaff$ into a direct sum of irreducible modules for
$V^\gm = V_{\Z\gm}$. Any irreducible module for $V_{\Z\gm}$ is isomorphic
to one of $V_{\Z\gm + n\gm/2k}$, $0 \le n \le 2k-1$ \cite{Dong}. Let $V^\gm
\cdot v$ be the $V^\gm$-submodule of $V_{L^\perp}$ generated by an element
$v$ of $V_{L^\perp}$. Then $V^\gm \cdot (E_{-1})^j\1$ (resp. $V^\gm \cdot
(F_{-1})^j\1$) is isomorphic to $V_{\Z\gm + j\gm/k}$ (resp. $V_{\Z\gm -
j\gm/k}$). Now, the action of $H_0 = \gamma(0)$ on $e^{n\gm/2k}$ is given
by $H_0 e^{n\gm/2k} = n e^{n\gm/2k}$. Moreover, the eigenvalues of $H_0$ on
$\Vaff$ are even integers, since $h(0)u = 2(q-r)u$ for a vector $u$ of the
form \eqref{eq:V-basis} in $V(k,0)$. Hence $V_{\Z\gm + n\gm/2k}$ does not
appear as a direct summand in $\Vaff$ unless $n$ is even. Let
\begin{equation*}
M^{0,j} = \{v \in \Vaff\,|\, H_{m} v = -2j\delta_{m,0}v \text{ for } m\ge
0\}
\end{equation*}
for $0 \le j \le k-1$. Then $\Vaff = \oplus_{j=0}^{k-1} (V^\gm \cdot
(F_{-1})^j\1) \otimes M^{0,j}$ as $V^\gm$-modules. That is, the following
lemma holds.
\begin{lem}\label{lem:dec-0}
$L(k,0) = \oplus_{j=0}^{k-1} V_{\Z\gm - j\gm/k} \otimes M^{0,j}$
as $V_{\Z\gm}$-modules.
\end{lem}

In the case $j=0$, $M^{0,0}$ coincides with the commutant $K_0$ of
$V_{\wh}(k,0)$ in $L(k,0)$. The restriction of $\theta$ to $M^{0,0}$, which
we denote by the same symbol $\theta$, is an automorphism of $M^{0,0}$ of
order $2$.

In order to describe irreducible $M^{0,0}$-modules contained in
$V_{L^\perp}$, let
\begin{equation*}
v^i = \sum_{\substack{I \subset \{1,2,\ldots, k\}\\|I|=i}} e^{\al_I/2},
\qquad v^{i,j} = \frac{1}{j\,!}(F_0)^j v^i
\end{equation*}
for $0 \le i \le k$, $0 \le j \le i$. Note that $v^0$ is the vacuum vector
$\1$ and $v^{i,0} = v^i$. In fact\footnote{The corresponding equation on
page 31 of \cite{DLY2} is incorrect.},
\begin{equation*}
v^{i,j} = \sum_{\substack{I \subset \{1,2,\ldots, k\}\\|I|=i}}
\sum_{\substack{J \subset I\\
|J| = j}} e^{\al_{I-J}/2 - \al_J/2}.
\end{equation*}
From this explicit form of $v^{i,j}$, we see that $\theta v^{i,j} =
v^{i,i-j}$, $0 \le i \le k$, $0 \le j \le i$.

We have $H_0 v^{i,j} = (i-2j)v^{i,j}$, $E_0 v^{i,0} =0$, $E_0 v^{i,j} =
(i-j+1)v^{i,j-1}$ for $1 \le j \le i$, $F_0 v^{i,j} = (j+1)v^{i,j+1}$ for
$0 \le j \le i-1$, $F_0 v^{i,i} = 0$, and $H_n v^{i,j} = E_n v^{i,j} = F_n
v^{i,j} = 0$ for $n \ge 1$. Hence the subspace $U^i$ of $V_{L^\perp}$
spanned by $v^{i,j}$, $0 \le j \le i$ is an $i+1$ dimensional irreducible
module for $\spn\{ H_0, E_0, F_0\} \cong sl_2$. Furthermore, the
$\Vaff$-submodule $\Vaff\cdot v^i$ of $V_{L^\perp}$ generated by $v^i$ is
isomorphic to an irreducible $L(k,0)$-module $L(k,U^i)$ with top level
$U^i$ \cite{FZ}. If $i = 0$, then $U^0 = \C\1$ and $L(k,U^0)$ coincides
with $L(k,0)$. We denote $L(k,U^i)$ by $L(k,i)$ for a general $i$ also and
identify it with $\Vaff\cdot v^i$. Let
\begin{equation*}
M^{i,j} = \{ v \in \Vaff\cdot v^i\,|\, H_m v = (i-2j)\delta_{m,0}v \text{
for } m\ge 0\}
\end{equation*}
for $0 \le i \le k$, $0 \le j \le k - 1$. Each $M^{i,j}$ is an
$M^{0,0}$-module. We decompose $\Vaff\cdot v^i$ into a direct sum of
irreducible $V^\gm$-modules and obtain the following lemma, which is a
generalization of Lemma \ref{lem:dec-0}.
\begin{lem}\label{lem:dec-LUi}
$L(k,i) = \oplus_{j=0}^{k-1} V_{\Z\gm + (i-2j)\gm/2k} \otimes M^{i,j}$ as
$V_{\Z\gm}$-modules.
\end{lem}

Next, we consider an isomorphism $\sigma = \exp(2\pi\sqrt{-1}\gamma(0)/2k)$
induced by the action of $\gamma(0)$ on $V_{L^\perp}$. Recall that $h(0)u =
2(q-r)u$ for $u \in V(k,0)$ being as in \eqref{eq:V-basis} and that $H_0
v^{i,j} = (i-2j)v^{i,j}$. Thus there are exactly $k$ distinct eigenvalues
$i-2j$, $0 \le j \le k-1$ modulo $2k$ of the action of $\gamma(0) = H_0$ on
$L(k,i)$. Hence $\sigma$ has $k$ distinct eigenvalues
$\exp(2\pi\sqrt{-1}(i-2j)/2k)$, $0 \le j \le k-1$ on $L(k,i)$. Let $L(k,i)
= \oplus_{j=0}^{k-1} L(k,i)^j$ be the eigenspace decomposition, where
$L(k,i)^j$ is the eigenspace for $\sigma$ on $L(k,i)$ with eigenvalue
$\exp(2\pi\sqrt{-1}(i-2j)/2k)$. Since $\la \gamma, (i-2j)\gamma/2k \ra =
i-2j$, Lemma \ref{lem:dec-LUi} implies that
\begin{equation}\label{eq:eigenspace-LUi}
L(k,i)^j = V_{\Z\gm + (i-2j)\gm/2k} \otimes M^{i,j}.
\end{equation}

For convenience, we understand the second parameter $j$ of $M^{i,j}$ to be
modulo $k$. We study some basic properties of $M^{0,0}$-modules $M^{i,j}$,
$0 \le i \le k$, $0 \le j \le k - 1$.

\begin{thm}\label{thm:properties-Mij}
$(1)$ $M^{i,j}$ is an irreducible $M^{0,0}$-module.

$(2)$ $M^{i,j} \cong M^{k-i,k-i+j}$ as $M^{0,0}$-modules.

$(3)$ The automorphism $\theta$ of $M^{0,0}$ induces a permutation $M^{i,j}
\mapsto M^{i,i-j}$ on these irreducible $M^{0,0}$-modules.

$(4)$ The top level of $M^{i,j}$ is a one dimensional space $\C v^{i,j}$,
$0 \le i \le k$, $0 \le j \le \min\{i, k-1\}$.
\end{thm}

\begin{proof}
We first show the assertion (1). By Lemma \ref{lem:dec-0}, $L(k,0) =
\oplus_{j=0}^{k-1} V_{\Z\gm - j\gm/k} \otimes M^{0,j}$. Suppose $M^{0,j}$
is not an irreducible $M^{0,0}$-module and let $U$ be a proper submodule of
$M^{0,j}$. Take $0 \ne u \in U$. Let $S = L(k,0) \cdot u$ be the subspace
of $L(k,0)$ spanned by $v_n u$, $v \in L(k,0)$, $n \in \Z$, which is an
ideal of $L(k,0)$. Actually, $L(k,0)$ is equal to $S$, since it is a simple
vertex operator algebra. Now, let $v \in V_{\Z\gm - r\gm/k} \otimes
M^{0,r}$. Then the fusion rules $V_{\Z\gm + a} \times V_{\Z\gm + b} =
V_{\Z\gm + a + b}$ of irreducible $V_{\Z\gamma}$-modules \cite[Chapter
12]{DL} imply that $v_n u$ lies in $V_{\Z\gm - (r+j)\gm/k} \otimes
M^{0,r+j}$. Hence $S \cap (V_{\Z\gm - j\gm/k} \otimes M^{0,j}) = V_{\Z\gm -
j\gm/k} \otimes U$, which contradicts the fact that $L(k,0)=S$. Thus
$M^{0,j}$ is an irreducible $M^{0,0}$-module. We apply a similar argument
to the irreducible $L(k,0)$-module $L(k,i)$ and use the decomposition in
Lemma \ref{lem:dec-LUi}. Then we obtain (1).

Next, we show the assertion (2). A multiplication by $e^{-\gamma/2}$ on the
group algebra $\C[L^\perp]$ induces a linear isomorphism $\psi : u \otimes
e^\beta \mapsto u \otimes e^{\beta - \gamma/2}$, $u \in M(1)$, $\beta \in
L^\perp$ of $V_{L^\perp} = M(1) \otimes \C[L^\perp]$ onto itself. Since
$M^{0,0}$ is contained in $V_{L'}$ by Proposition \ref{prop:VL-prime}, it
follows from \eqref{eq:L-prime} and the definition of
$Y_{V_{L^\perp}}(v,x)$ that $\psi$ commutes with the vertex operator
$Y_{V_{L^\perp}}(v,x)$ associated with any $v \in M^{0,0}$. Thus $\psi$ is
an isomorphism of $M^{0,0}$-modules. We have $\psi( V_{\Z\gamma
+(i-2j)\gamma/2k}) = V_{\Z\gamma + ((k - i) - 2(k - i + j))\gamma/2k}$.
Hence (2) holds.

We show the assertion (3). Since $\theta v^{i,j} = v^{i,i-j}$, $U^i$ is
invariant under $\theta$. Moreover, $\theta(H_m v) = -H_m \theta(v)$ for $m
\in \Z$ and $v \in \Vaff\cdot v^i$, since $\theta H = -H$. Thus (3) holds.
Actually, $\theta$ maps $ V_{\Z\gamma +(i-2j)\gamma/2k}$ onto  $V_{\Z\gamma
+(i-2(i-j))\gamma/2k}$.

Finally, we show the assertion (4). The top level of the irreducible
$L(k,0)$-module $L(k,i)$ is $U^i = \oplus_{j=0}^i \C v^{i,j}$. Moreover,
$\gamma(0)v^{i,j} = (i-2j)v^{i,j}$ and so $v^{i,j} \in L(k,i)^j$, $0 \le j
\le \min\{i, k-1\}$. That is, the eigenvalues of $\sigma$ on the
eigenvectors $v^{i,j}$, $0 \le j \le \min\{i, k-1\}$ are all different. The
only exception is the case $i=k$ and $j=0,k$. Hence the top level of
$L(k,i)^j$ is a one dimensional space $\C v^{i,j}$ unless $i=k$ and $j=0$.
In such a case the top level of $M^{i,j}$ is also $\C v^{i,j}$ by
\eqref{eq:eigenspace-LUi}. As to the top level of $M^{k,0}$, recall that
$v^{k,0} = e^{\gamma/2}$ and $v^{k,k} = e^{-\gamma/2}$, on which $\sigma$
acts as $-1$. Thus the top level of $L(k,k)^0$ is a two dimensional space
$\C v^{k,0} + \C v^{k,k}$, which coincides with the top level of
$V_{\Z\gamma+\gamma/2}$. Therefore, the top level of $M^{k,0}$ is one
dimensional by \eqref{eq:eigenspace-LUi}. This proves (4). Note that
$M^{k,0} \cong M^{0,0}$ as $M^{0,0}$-modules by the assertion (2). Indeed,
the isomorphism $\psi$ of $M^{0,0}$-modules maps $v^{k,0}$ to the vacuum
vector $\1$ and $\1$ to $v^{k,k}$.
\end{proof}

One may prove the irreducibility of $M^{i,j}$ in a different manner.
Indeed, each irreducible $L(k,0)$-module $L(k,i)$ is $\sigma$-stable, for
$\sigma$ is an inner automorphism. We see from \eqref{eq:eigenspace-LUi}
that the space $L(k,0)^{\la \sigma \ra}$ consisting of the elements of
$L(k,0)$ fixed by $\sigma$ is $V_{\Z\gamma} \otimes M^{0,0}$. By
\cite[Theorem 5.4]{DY}, the eigenspace $L(k,i)^j$ for $\sigma$ is an
irreducible $L(k,0)^{\la \sigma \ra}$-module. Hence
\eqref{eq:eigenspace-LUi} implies that $M^{i,j}$ is irreducible as a module
for $M^{0,0}$. In the proof of Theorem \ref{thm:properties-Mij} (1), we
have shown that $v_n u \in V_{\Z\gm - (r+j)\gm/k} \otimes M^{0,r+j}$ by
using the fusion rules of $V_{\Z\gamma}$. This fact comes from
\eqref{eq:eigenspace-LUi} also. The isomorphism of $M^{0,0}$-modules
discussed in Theorem \ref{thm:properties-Mij} (2) is a special case of
\cite[Corollary 5.7]{DLM2} (see \cite[Remark 5.8]{DLM2} also). That is, it
is a kind of spectrum flow.

The character of $M^{i,j}$ is known \cite{CGT, GQ}. In fact,
\begin{equation*}
\ch M^{i,j} = \eta(\tau) c^i_{i-2j}(\tau)
\end{equation*}
by \cite[(3.34)]{GQ}. Note that $k$, $l$ and $m$ of \cite{GQ} are $k$, $i$
and $i-2j$, respectively in our notation. Then \cite[(3.36), (3.40)]{GQ}
mean $c^i_{i-2j} = c^i_{i-2(j-k)}$, $c^i_{i-2j} = c^i_{2j-i}$ and
$c^i_{i-2j} = c^{k-i}_{k+i-2j}$. As to the first equation, recall that our
$j$ of $M^{i,j}$ is considered to be modulo $k$. The remaining two
equations are compatible with Theorem \ref{thm:properties-Mij} (3) and (2),
respectively.

\begin{prop}\label{prop:action-on-vij}
The weight $0$ operators $o(\omega) = \omega_1$, $o(W^p) = W^p_{p-1}$,
$p=3,4,5$ act on $v^{i,j}$, $0 \le i \le k$, $0 \le j \le i$ as follows.
\begin{equation*}
o(\om)v^{i,j} = \frac{1}{2k(k+2)}\Big( k(i-2j) - (i-2j)^2 + 2k(i-j+1)j
\Big) v^{i,j},
\end{equation*}
\begin{equation*}
o(W^3) v^{i,j} = \Big( k^2 (i-2j) - 3k(i-2j)^2 + 2(i-2j)^3 -
6k(i-2j)(i-j+1)j \Big) v^{i,j},
\end{equation*}
\begin{equation*}
\begin{split}
o(W^4) v^{i,j} &=  \Big( 2k^2(k^2+k+1)(i-2j) -
k(13k^2+8k+2)(i-2j)^2\\
& \quad + 2k(11k+6)(i-2j)^3 - (11k+6)(i-2j)^4\\
& \quad + 4k^2(k-3)(k-2)(i-j+1)j - 4k^2(6k-5)(i-2j)(i-j+1)j\\
& \quad + 4k(11k+6)(i-2j)^2(i-j+1)j\\
& \quad -2k^2(6k-5)(i-j+1)(i-j+2)(j-1)j \Big) v^{i,j},
\end{split}
\end{equation*}
\begin{equation*}
\begin{split}
o(W^5) v^{i,j} &= \Big(- 2k^3(k^2+3k+5)(i-2j) + 5k^2(5k^2+6k+6)(i-2j)^2\\
& \quad - 20k(4k^2+3k+1)(i-2j)^3 + 5k(19k+12)(i-2j)^4\\
& \quad - 2(19k+12)(i-2j)^5\\
& \quad + 10k^2(5k^2-14k+20)(i-2j)(i-j+1)j\\
& \quad - 20k^2(10k-7)(i-2j)^2(i-j+1)j\\
& \quad + 10k(19k+12)(i-2j)^3(i-j+1)j\\
& \quad - 10k^2(10k-7)(i-2j)(i-j+1)(i-j+2)(j-1)j \Big) v^{i,j}.
\end{split}
\end{equation*}
\end{prop}

\begin{proof}
We use the representation of $\widehat{sl}_2$ given by the correspondence
\eqref{eq:HhEeFf}, namely

(i) $h(0)v^{i,j} = (i-2j)v^{i,j}$,

(ii) $e(0)v^{i,0} = 0$, $e(0)v^{i,j} = (i-j+1)v^{i,j-1}$ for $1 \le j \le
i$,

(iii) $f(0)v^{i,i} = 0$, $f(0)v^{i,j} = (j+1)v^{i,j+1}$ for $0 \le j \le
i-1$,

(iv) $a(n)v^{i,j} = 0$ for $a \in \{h,e,f\}$, $n \ge 1$.

Recall the vertex operator $Y(u,x)$ of the $L(k,0)$-module $L(k,i)$. For $a
\in \{h,e,f\}$ and $m \ge 1$,
\begin{equation*}
Y(a(-m)\1,x) = \frac{1}{(m-1)!} \Big( \frac{d}{dx}\Big)^{m-1} a(x).
\end{equation*}

Hence the coefficient $o(a(-m)\1)$ of $x^{-m}$ in $Y(a(-m)\1,x)$ is
$(-1)^{m-1}a(0)$. Let $a^1, \ldots, a^r \in \{ h,e,f\}$ and $m_1,\ldots,
m_r \ge 1$. The vertex operator $Y(u,x)$ associated with the vector $u =
a^1(-m_1) \cdots a^r(-m_r)\1$ is given recursively by
\begin{align*}
Y(u,x) &= \Big( \frac{1}{(m_1-1)!} \Big( \frac{d}{dx}\Big)^{m_1-1} a^1(x)^-
\Big) Y(a^2(-m_2) \cdots a^r(-m_r)\1,x) \\
& \qquad \quad + Y(a^2(-m_2) \cdots a^r(-m_r)\1,x) \Big( \frac{1}{(m_1-1)!}
\Big( \frac{d}{dx}\Big)^{m_1-1} a^1(x)^+ \Big),
\end{align*}
where $a(x)^- = \sum_{n < 0} a(n)x^{-n-1}$ and $a(x)^+ = \sum_{n \ge 0}
a(n)x^{-n-1}$. The operator $o(u)$ is the coefficient of $x^{-m_1 - \cdots
- m_r}$ in $Y(u,x)$. Since $a(n) v^{i,j} = 0$ for $n \ge 1$, we have
\begin{equation*}
o(a^1(-m_1) \cdots a^r(-m_r)\1) v^{i,j} = (-1)^{m_1 + \cdots + m_r - r}
a^r(0) \cdots a^1(0) v^{i,j}.
\end{equation*}

For instance, $o(h(-2)\1)$, $o(h(-1)^2\1)$ and $o(e(-1)f(-1)\1)$ act on
$v^{i,j}$ as $-(i-2j)$, $(i-2j)^2$ and $(i-j+1)j$, respectively. Therefore,
we obtain $o(\omega)v^{i,j}$. The action of $o(W^3)$, $o(W^4)$ and $o(W^5)$
on $v^{i,j}$ can be calculated similarly.
\end{proof}

Let $a(i,j)$ and $b(i,j)$ be the eigenvalues of the operators $o(\omega)$
and $o(W^3)$ on $v^{i,j}$ given in Proposition \ref{prop:action-on-vij},
respectively. We can verify that $a(i,i-j) = a(i,j)$ and $b(i,i-j) =
-b(i,j)$. These relations are compatible with Theorem
\ref{thm:properties-Mij} (3), since $\theta$ fixes $\omega$ and maps $W^3$
to its negative. Similar relations hold for the eigenvalues of $o(W^4)$ and
$o(W^5)$ on $v^{i,j}$.

We close this section with the following conjecture.

\begin{conjecture}\label{Conj:W-properties}
$(1)$ $M^{i,j}$, $0 \le i \le k$, $0 \le j \le i$ are not isomorphic each
other except the isomorphisms $M^{i,i} \cong M^{k-i,0}$. There are
$k(k+1)/2$ inequivalent irreducible $M^{0,0}$-modules among those
$M^{i,j}$'s.

$(2)$ $M^{0,0} = \W$ and there are exactly $k(k+1)/2$ inequivalent
irreducible $M^{0,0}$-modules, which are represented by $M^{i,j}$, $0 \le i
\le k$, $0 \le j \le i-1$. Furthermore, Zhu's algebra $A(M^{0,0})$ of
$M^{0,0}$ is generated by $[\omega]$ and $[W^3]$ if $k \ge 3$.

$(3)$ The vertex operator algebra $M^{0,0}$ is rational and $C_2$-cofinite.
\end{conjecture}

\section{Case $k \le 6$}\label{Sect:small-k-case}
In this section we will show that Conjecture \ref{Conj:W-properties} is
true for $k \le 6$. Indeed, one may verify that there are $k(k+1)/2$
different pairs of eigenvalues of the operators $o(\omega)$ and $o(W^3)$ on
the top level $\C v^{i,j}$ of $M^{i,j}$, $0 \le i \le k$, $0 \le j \le i-1$
once $k$ is given (see Proposition \ref{prop:action-on-vij}). That is,
$o(\omega)$ and $o(W^3)$ are expected to be sufficient to distinguish
inequivalence of those irreducible $M^{0,0}$-modules $M^{i,j}$'s. This is
the case if $k$ is a small positive integer, say $k \le 6$. In this way we
see that the assertion (1) of Conjecture \ref{Conj:W-properties} is true
for $k \le 6$.

The singular vector discussed in Section \ref{Sect:maximal-ideal-tI} plays
a crucial role in the proof of the remaining assertions of Conjecture
\ref{Conj:W-properties}. Let $\bu^0 = f(0)^{k+1}e(-1)^{k+1}\1$. By Theorem
\ref{thm:singular-vec}, $\bu^0 \in \tI$ and so $\bu^0 = 0$ in $K_0 =
M^{0,0}$. Our argument is based on a detailed analysis of the vector
$\bu^0$. First of all, we express $\bu^0$ as a linear combination of the
basis \eqref{eq:V-basis} of $V(k,0)$. The expression enables us to write
$\bu^0$ as a linear combination of the vectors of normal form
\eqref{eq:normal-form} of weight $k+1$. This in particular implies that
$\tW$ contains $\bu^0$. Unfortunately, we do not succeed in handling this
process for a general $k$. It seems difficult even to show that $\bu^0 \in
\tW$. Therefore, we discuss only the case $k \le 6$ in this section.
Actually, $\bu^0$ is a scalar multiple of $W^3$, $W^4$ or $W^5$ in the case
$k = 2$, $3$ or $4$. In such a degenerate case, $\W$ is isomorphic to a
well-known vertex operator algebra (see below for details). Thus we
concentrate on the cases $k=5$ and $6$.

We study Zhu's algebra $A(\W)$ of $\W$ for the classification of
irreducible $\W$-modules. It turns out that the null fields $\bv^0$ and
$\bv^1$ considered in Section \ref{Sect:W2345} and $\bu^r = (W^3_1)^r
\bu^0$, $r = 0,1,2,3$ are sufficient to determine $A(\W)$ in the case
$k=5$, $6$. Once all irreducible $\W$-modules are known, we can show that
$\W = M^{0,0}$. One more null field $\bv^2$ is necessary for the proof of
the $C_2$-cofiniteness of $\W$. Finally, we use \cite[Proposition
5.11]{DLTYY} to establish the rationality of the vertex operator algebra
$\W$.

\subsection{Case $k=2$}\label{subsec:k-is-2}
In this case $\bu^0$ is a scalar multiple of $W^3$. In fact, we have $\bu^0
= -3W^3$. Thus $W^3 \in \tI$. Now, $W^3_1 W^3 = (72/7)W^4$ by
\eqref{eq:W3m1W3}, for we are assuming that $k=2$. Hence $W^4 \in \tI$.
Then \eqref{eq:W3m1W4} implies that $W^5 \in \tI$. Therefore, $W^3$, $W^4$
and $W^5$ become $0$ in $M^{0,0}$. The vertex operator algebra $\W$ is
generated by the conformal vector $\omega$ and it is isomorphic to a simple
Virasoro vertex operator algebra $\CL(1/2,0)$ of central charge $1/2$. It
is well known that $\CL(1/2,0)$ has exactly three irreducible modules
$\CL(1/2,h)$, $h = 0, 1/2, 1/16$, where $\CL(c,h)$ denotes an irreducible
highest weight module with highest weight $h$ for a Virasoro algebra of
central charge $c$. Actually, $M^{0,0}$, $M^{2,1}$ and $M^{1,0}$ are
isomorphic to those irreducible modules, respectively. Moreover, we have
$\W = M^{0,0}$. Thus Conjecture \ref{Conj:W-properties} is true for $k=2$.

\subsection{Case $k=3$}\label{subsec:k-is-3}
In this case $\bu^0 = -(8/13)W^4$ is a scalar multiple of $W^4$. Thus $W^4
\in \tI$. Then \eqref{eq:W3m1W4} with $k = 3$ implies that $W^5 \in \tI$.
Hence $W^4$ and $W^5$ become $0$ in $M^{0,0}$. The vertex operator algebra
$\W$ is isomorphic to a three state Potts model $\CL(4/5,0) \oplus
\CL(4/5,3)$. It is known that a three state Potts model has exactly six
irreducible modules \cite{KMY}. Moreover, we have $\W = M^{0,0}$. The
results in \cite{KMY} agree with the assertions of Conjecture
\ref{Conj:W-properties}. The vertex operator algebra $V_{L'} \cong
V_{\sqrt{2}A_2}$ was studied in detail \cite{KMY}. For the relationship
between $M^{0,0}$ and a three state Potts model, see \cite[Section 5]{CGT},
\cite[Appendix B]{GQ}.

\subsection{Case $k=4$}\label{subsec:k-is-4}
In this case $\bu^0 = (15/22)W^5$ is a scalar multiple of $W^5$ and so
$W^5$ becomes $0$ in $M^{0,0}$. The vertex operator algebra $\W$ is
isomorphic to $V_{\Z\beta}^+$ with $\la\beta,\beta\ra = 6$, which has
exactly ten irreducible modules \cite{DN}. Moreover, we have $\W =
M^{0,0}$. The results in \cite{DN} agree with the assertions of Conjecture
\ref{Conj:W-properties}.\footnote{In the table of \cite[page 186]{DN} the
action of $J$ on the top level of $V^\pm_{L+\alpha/2}$ should read
$k^2/4-k/4$.} The vertex operator algebra $V_{L'} \cong V_{\sqrt{2}A_3}$
was studied in detail \cite{DLY}.

\subsection{Case $k=5$}\label{subsec:k-is-5}
Let $\bv^0$ and $\bv^1$ be as in Section 2. In addition to these two null
fields of $\tW$, we consider the image of $\bu^0$ under the operator
$W^3_1$ successively, that is, $\bu^r = (W^3_1)^r \bu^0$, $r = 1, 2, 3$.
The weight of $\bu^r$ is $k + 1 + r$. We first express $\bu^r$ as a linear
combination of the basis \eqref{eq:V-basis} of $V(k,0)$, and then express
it as a linear combination of the vectors of normal form
\eqref{eq:normal-form}. For instance,
\begin{equation*}
\begin{split}
\bu^0 &=
-(56260915200/97)\omega_{-5}\1 -(47822745600/97)\omega_{-3}\omega_{-1}\1\\
& \quad + (43180603200/97)(\omega_{-2})^2\1
+ (33230937600/97)(\omega_{-1})^3\1\\
& \quad -(4032/5)(W^3_{-1})^2\1 + (550368/97)\omega_{-1}W^4_{-1}\1 +
(340704/97)W^4_{-3}\1.
\end{split}
\end{equation*}

This equation is obtained from the expression of the vectors $\bu^0$,
$\omega_{-5}\1$, $\omega_{-3}\omega_{-1}\1$, $(\omega_{-2})^2\1$,
$(\omega_{-1})^3\1$, $(W^3_{-1})^2\1$, $\omega_{-1}W^4_{-1}\1$ and
$W^4_{-3}\1$ as linear combinations of the basis \eqref{eq:V-basis} of
$V(k,0)$. The above expression of $\bu^0$ implies that $\bu^0 \in \tW^+$.
As to $\bu^1$, $\bu^2$ and $\bu^3$, see Appendix
\ref{App:Case-k-5-singular-vec}.

Next, take the image in Zhu's algebra $A(\tW) = \tW/O(\tW)$ of the right
hand side of the expression of $\bu^r$ as a linear combination of the
vectors of normal form. Then we can express the image $[\bu^r]$ of $\bu^r$
in $A(\tW)$ as a linear combination of elements of the form
\eqref{eq:Zhu-element}. We then replace $[\omega]$, $[W^3]$, $[W^4]$ and
$[W^5]$ with $w_2$, $w_3$, $w_4$ and $w_5$, respectively in the expression
of $[\bu^r]$. Let $P_r$, $r=0,1,2,3$ be the polynomial in $w_2$, $w_3$,
$w_4$, $w_5$ obtained from $\bu^r$ in this manner. Actually, we multiply it
by a suitable integer. Then we have
\begin{align*}
P_0 &= 82418000 w_2^3-36225000 w_2^2+(1365 w_4+2530000) w_2-194
w_3^2-130 w_4,\\
P_1 &= (5116834800 w_2^2-3289532400 w_2-49959 w_4+190779600) w_3\\
& \quad -3354260 w_2 w_5 +479180 w_5,\\
P_2 &= -519970178910210000 w_2^4 + 301201024956142500 w_2^3\\
& \quad +(- 8403180446500 w_4 - 47718693405180000) w_2^2\\
& \quad +(1231205050775 w_3^2 + 1880038332025 w_4 + 2220670158630000) w_2\\
& \quad -180544972860 w_3^2 + 33957081 w_3 w_5 + 1437404 w_4^2
-102193394550 w_4,\\
P_3 &= -46312512741411 w_3^3\\
& \quad +(8531538341629506000 w_2^3 - 737916475955662500 w_2^2\\
& \qquad +(433503066092460 w_4 - 286997147877132000) w_2\\
& \qquad -37952176698930 w_4 + 27372745589112000) w_3\\
& \quad +6990074602966000 w_2^2 w_5 - 1318615129549900 w_2 w_5\\
& \quad -3246519796 w_4 w_5 +53681912466000 w_5.
\end{align*}

Since $\W$ is a homomorphic image of $\tW$, Zhu's algebra $A(\W)$ of $\W$
is a homomorphic image of $A(\tW)$. Take the composition with the
surjective homomorphism $\tvp : \C[w_2,w_3,w_4,w_5] \rightarrow A(\tW)$ of
associative algebra considered in Section \ref{Sect:W2345}. Then we obtain
a surjective homomorphism $\varphi : \C[w_2,w_3,w_4,w_5] \rightarrow
A(\W)$. The kernel of $\varphi$ contains the polynomials $Q_0$ and $Q_1$
studied in Section \ref{Sect:W2345} with $k = 5$, for $\bv^0 = \bv^1 = 0$
in $V(k,0)$. The kernel also contains the above four polynomials $P_0$,
$P_1$, $P_2$ and $P_3$, for $\bu^0$, $\bu^1$, $\bu^2$ and $\bu^3$ lie in
$\tI$. We can verify that a Gr\"{o}bner basis of the ideal $\CP$ of
$\C[w_2,w_3,w_4,w_5]$ generated by $P_0$, $P_1$, $P_2$, $P_3$, $Q_0$ and
$Q_1$ with $k = 5$ consists of the five polynomials
\begin{align*}
R_1 &= w_2(5w_2-6)(5w_2-4)(7w_2-6)(7w_2-2)\\
& \qquad \quad \cdot
(35w_2-23)(35w_2-17)(35w_2-3)(35w_2-2),\\
R_2 &=
w_3(5w_2-6)(5w_2-4)(35w_2-23)(35w_2-17)(35w_2-3)(35w_2-2),\\
R_3 &= p(w_2) + 564841728w_3^2,\\
R_4 &= q(w_2) + 14685884928w_4,\\
R_5 &= r(w_2)w_3 + 5575284w_5,
\end{align*}
where $p(w_2)$, $q(w_2)$ and $r(w_2)$ are polynomials in $w_2$ of degree
$8$, $8$ and $5$, respectively. The common factor of $R_1$ and $p(w_2)$ is
$w_2(7w_2-6)(7w_2-2)$ and that of $R_1$ and $q(w_2)$ is $w_2$, while
$r(w_2)$ has no common factor with $R_1$. The Gr\"{o}bner basis implies
that $\C[w_2,w_3,w_4,w_5]/\CP$ is a $15$ dimensional space with basis
$w_2^m + \CP$, $0 \le m \le 8$, $w_2^n w_3 + \CP$, $0 \le n \le 5$. In
particular, $\C[w_2,w_3,w_4,w_5]/\CP$ is generated by $w_2 + \CP$ and $w_3
+ \CP$.

We do not show that $\W = M^{0,0}$ so far. Hence the $15$ inequivalent
irreducible $M^{0,0}$-modules $M^{i,j}$ constructed in Section
\ref{Sect:VOA-W} may not be irreducible nor inequivalent as $\W$-modules.
Let $N^{i,j}$ be the $\W$-submodule of $M^{i,j}$ generated by $v^{i,j}$, so
that the top level of $N^{i,j}$ is a one dimensional space $\C v^{i,j}$.
Then $N^{i,j}$ has a unique maximal submodule, possibly $0$. Consider the
quotient module $U^{i,j}$ of $N^{i,j}$ by its unique maximal submodule. It
is an irreducible $\W$-module with top level $\C v^{i,j}$. By Proposition
\ref{prop:action-on-vij}, we know how $o(\omega)$, $o(W^3)$, $o(W^4)$ and
$o(W^5)$ act on $\C v^{i,j}$. We can verify that the $15$ quartets of the
eigenvalues of these four operators on $\C v^{i,j}$ are all different and
that they agree with the solutions $(w_2,w_3,w_4,w_5)$ of a system of
equations
\begin{equation}\label{eq:system-of-equations5}
R_1 = R_2 = R_3 = R_4 = R_5 = 0.
\end{equation}

By \cite[Theorem 2.2.2]{Zhu}, we conclude that $A(\W) \cong
\C[w_2,w_3,w_4,w_5]/\CP$ and any irreducible $\W$-module is isomorphic to
one of $U^{i,j}$'s. Furthermore, $\C[w_2,w_3,w_4,w_5]/\CP$ is semisimple,
for the system of equations \eqref{eq:system-of-equations5} has no multiple
root. Hence $A(\W)$ is semisimple. Note that $A(\W)$ is generated by
$[\omega]$ and $[W^3]$. This is consistent with the fact that the $15$
pairs of the eigenvalues of $o(\omega)$ and $o(W^3)$ on $\C v^{i,j}$ are
all different. That is, $o(\omega)$ and $o(W^3)$ are sufficient to
distinguish $U^{i,j}$'s.

Now, suppose $\W \ne M^{0,0}$ and consider the quotient $\W$-module
$M^{0,0}/\W$. It has integral weights. We can easily verify that the weight
$n$ subspace of $\W$ coincides with that of $M^{0,0}$ for a small $n$, say
$n = 0,1,2$. Therefore, the weight of any irreducible quotient of
$M^{0,0}/\W$ is greater that $2$. This is a contradiction since the weight
of the top level of $U^{i,j}$ is at most $6/5$ and $0$ is the only integral
one. Thus $\W = M^{0,0}$ and the assertion (2) of Conjecture
\ref{Conj:W-properties} is true for $k=5$.

It remains to prove the $C_2$-cofiniteness and the rationality of $\W$. We
have studied $\bu^r$, $r=0,1,2,3$ and the null fields $\bv^0$ and $\bv^1$
modulo $O(\tW)$ for the determination of Zhu's algebra of $\W$. As to the
proof of the $C_2$-cofiniteness, we consider $\bu^r$, $r=0,1,2,3$ and the
null fields $\bv^0$, $\bv^1$ and $\bv^2$ modulo $C_2(\tW)$. In fact,
$\bu^r$, $r=0,1,2,3$, $\bv^0$ and $\bv^1$ are not sufficient to show the
$C_2$-cofiniteness. Take the image in $\tW/C_2(\tW)$ of the right hand side
of the expression of $\bu^r$ as a linear combination of vectors of normal
form given in Appendix \ref{App:Case-k-5-singular-vec}. It is a polynomial
in $\omega + C_2(\tW)$ and $W^i + C_2(\tW)$, $i=3,4,5$. Replace $\omega +
C_2(\tW)$ and $W^i + C_2(\tW)$ with $x_2$ and $x_i$, $i=3,4,5$,
respectively in the polynomial and multiply it by a suitable integer. Let
$A_r$, $r=0,1,2,3$ be the polynomial obtained from $\bu^r$ in this manner.
Then
\begin{align*}
A_0 &=
82418000 x_2^3+1365 x_4 x_2-194 x_3^2,\\
A_1 &=
730976400 x_3 x_2^2-479180 x_5 x_2-7137 x_4 x_3,\\
A_2 &=
519970178910210000 x_2^4+8403180446500 x_4 x_2^2-1231205050775 x_3^2 x_2\\
& \quad -33957081 x_5 x_3-1437404 x_4^2,\\
A_3 &=
8531538341629506000 x_3 x_2^3+6990074602966000 x_5 x_2^2\\
& \quad +433503066092460 x_4 x_3 x_2-46312512741411 x_3^3-3246519796 x_5 x_4.
\end{align*}

Let $C_2(\W)$ be the subspace of $\W$ spanned by $u_{-2}v$ with $u,v \in
\W$. Since $\W$ is a homomorphic image of $\tW$, there is a homomorphism
from $\tW/C_2(\tW)$ onto $\W/C_2(\W)$. Its composition $\rho$ with the
surjective homomorphism $\tilde{\rho} : \C[x_2,x_3,x_4,x_5] \rightarrow
\tW/C_2(\tW)$ discussed in Section \ref{Sect:W2345} is a homomorphism from
$\C[x_2,x_3,x_4,x_5]$ onto $\W/C_2(\W)$. The kernel of $\rho$ contains the
polynomials $B_0$, $B_1$ and $B_2$ studied in Section \ref{Sect:W2345} with
$k=5$. It also contains the above four polynomials $A_r$, $r=0,1,2,3$, for
$\bu^r$ is $0$ in $\W$. Let $\CA$ be the ideal of $\C[x_2,x_3,x_4,x_5]$
generated by $A_r$, $r=0,1,2,3$ and $B_s$, $s=0,1,2$ with $k=5$. We can
verify that a Gr\"{o}bner basis of $\CA$ consists of the eleven
polynomials
\begin{align*}
S_1 &= x_2^6,\\
S_2 &= x_3 x_2^4,\\
S_3 &= 2780750 x_2^5-29 x_3^2 x_2^2,\\
S_4 &= 378000 x_3 x_2^3-x_3^3,\\
S_5 &= 82418000 x_2^3+1365 x_4 x_2-194 x_3^2,\\
S_6 &= 33674025000 x_2^5+377 x_4 x_3^2,\\
S_7 &= 804763750000 x_2^4+61327280 x_3^2 x_2-2379 x_4^2,\\
S_8 &= 730976400 x_3 x_2^2-479180 x_5 x_2-7137 x_4 x_3,\\
S_9 &= 4633930000 x_2^4-28315 x_3^2 x_2-13 x_5 x_3,\\
S_{10} &= 2018093000 x_3 x_2^3+13 x_5 x_4,\\
S_{11} &= 173625253725000 x_2^5-377 x_5^2.
\end{align*}

From $S_1$, $S_4$, $S_7$ and $S_{11}$, we see that
$\C[x_2,x_3,x_4,x_5]/\CA$ is finite dimensional. This establishes the
$C_2$-cofiniteness of $\W$ for $k=5$.

The set of eigenvalues of the action of $o(\omega)$ on $\C v^{i,j}$,  $0
\le i \le 5$, $0 \le j \le i-1$ is
\begin{equation*}
\CE = \{ 0, 2/35, 3/35, 2/7, 17/35, 23/35, 6/7, 4/5, 6/5\}.
\end{equation*}
The difference of any two rational numbers in the set $\CE$ is not an
integer. Thus by a similar argument as in the proof of \cite[Lemma
5.13]{DLTYY}, we have that any $\W$-module generated by an irreducible
$A(\W)$-module is irreducible. Hence the vertex operator algebra $\W$ is
rational by \cite[Proposition 5.11]{DLTYY}. Thus the assertion (3) of
Conjecture \ref{Conj:W-properties} is true for $k=5$.

\subsection{Case $k=6$}\label{subsec:k-is-6}
The argument is essentially the same as in the case $k = 5$. The six
singular vectors $\bu^0$, $\bu^1$, $\bu^2$, $\bu^3$, $\bv^0$ and $\bv^1$
are sufficient for the determination of Zhu's algebra $A(\W)$, while we
need one more singular vector $\bv^2$ for the proof of the
$C_2$-cofiniteness of $\W$. Indeed, we have
\begin{equation*}
\begin{split}
\bu^0 &= -(1420529376000/55483) \omega_{-3}W^3_{-1}\1
+ (1356106752000/55483) (\omega_{-1})^2W^3_{-1}\1\\
& \quad + (19141808000/491) \omega_{-2}W^3_{-2}\1
- (2212337344000/55483) \omega_{-1}W^3_{-3}\1\\
& \quad + (2043429304000/55483) W^3_{-5}\1 - (33950/339) W^3_{-1}W^4_{-1}\1\\
& \quad - (4632320/491) \omega_{-1}W^5_{-1}\1 - (1995840/491) W^5_{-3}\1
\end{split}
\end{equation*}
in the case $k=6$. In particular, $\bu^0 \in \tW^-$.

Set $\bu^r = (W^3_1)^r \bu^0$, $r = 1,2,3$ as in Section
\ref{subsec:k-is-5}. We consider the polynomial $P_r \in
\C[w_2,w_3,w_4,w_5]$ obtained from $\bu^r$ in a similar manner as in the
case $k=5$. Let $\CP$ be the ideal of $\C[w_2,w_3,w_4,w_5]$ generated by
the six polynomials $P_r$, $r=0,1,2,3$, $Q_0$ and $Q_1$ with $k=6$. We
calculate that a Gr\"{o}bner basis of $\CP$ consists of the following five
polynomials.
\begin{align*}
R_1 &= w_2(2w_2-3)(3w_2-4)(4w_2-3)(4w_2-1)(6w_2-5)(12w_2-7)(12w_2-1)\\
& \qquad \quad \cdot (32w_2-23)(32w_2-3)(96w_2-101)(96w_2-41)(96w_2-5),\\
R_2 &= w_3(3w_2-4)(6w_2-5)(12w_2-7)(12w_2-1)\\
& \qquad \quad \cdot (32w_2-23)(96w_2-101)(96w_2-41)(96w_2-5),\\
R_3 &= p(w_2) + 2399941984319748410712448453175 w_3^2,\\
R_4 &= q(w_2) + 59998549607993710267811211329375 w_4,\\
R_5 &= r(w_2)w_3 + 171818959801082568975 w_5,
\end{align*}
where $p(w_2)$, $q(w_2)$ and $r(w_2)$ are polynomials in $w_2$ of degree
$12$, $12$ and $7$, respectively. The common factor of $R_1$ and $p(w_2)$
is $w_2(2w_2-3)(4w_2-3)(4w_2-1)(32w_2-3)$ and that of $R_1$ and $q(w_2)$ is
$w_2$, while $r(w_2)$ has no common factor with $R_1$. The Gr\"{o}bner
basis implies that $\C[w_2,w_3,w_4,w_5]/\CP$ is a $21$ dimensional space
with basis $w_2^m + \CP$, $0 \le m \le 12$, $w_2^n w_3 + \CP$, $0 \le n \le
7$.

In the case $k = 6$, we have $21$ inequivalent irreducible
$M^{0,0}$-modules $M^{i,j}$. Consider an irreducible subquotient $U^{i,j}$
of $M^{i,j}$ as in the case $k = 5$. The $21$ quartets of the eigenvalues
of $o(\omega)$, $o(W^3)$, $o(W^4)$ and $o(W^5)$ on $\C v^{i,j}$ are all
different and they agree with the solutions $(w_2,w_3,w_4,w_5)$ of a system
of equations $R_1 = R_2 = R_3 = R_4 = R_5 = 0$. Thus $A(\W)$ is isomorphic
to $\C[w_2,w_3,w_4,w_5]/\CP$. Moreover, it is semisimple and generated by
$[\omega]$ and $[W^3]$.

The set of eigenvalues of $o(\omega)$ on $\C v^{i,j}$, $0 \le i \le 6$, $0
\le j \le i-1$ is
\begin{equation*}
\CE = \{ 0, 5/96, 1/12, 1/4, 3/32, 41/96, 7/12, 3/4, 23/32, 101/96, 5/6,
4/3, 3/2\}.
\end{equation*}
Thus the weight of $v^{i,j}$ is at most $3/2$ and $0$ is the only integral
one. By a similar argument as in the case $k = 5$, we have that $\W =
M^{0,0}$ and the assertion (2) of Conjecture \ref{Conj:W-properties} holds.

For the proof of the $C_2$-cofiniteness of $\W$, we use $\bu^r$,
$r=0,1,2,3$ and $\bv^s$, $s=0,1,2$. Consider seven polynomials $A_r$,
$r=0,1,2,3$ and $B_s$, $s=0,1,2$ obtained in a similar manner as in the
case $k=5$. We can verify that the ideal $\CA$ generated by these seven
polynomials is of finite codimension in $\C[x_2,x_3,x_4,x_5]$. Thus $\W$ is
$C_2$-cofinite.

Let $U$ be an irreducible $A(\W)$-module. Then $U = \C u$ is one
dimensional and $L(0)u = \lambda u$ for some $\lambda \in \CE$. We want to
show that any $\W$-module $M$ generated by $U$ is irreducible (see
\cite[Lemma 5.13]{DLTYY}). If $\lambda \ne 5/96$, then $\CE \cap (\lambda +
\Z) = \{ \lambda \}$ and so there is no singular vector for $\W$ of weight
greater than $\lambda$ in $M$. By a similar argument as in the proof of
\cite[Lemma 5.13]{DLTYY}, we obtain that $M$ is an irreducible $\W$-module.

Suppose $\lambda = 5/96$. We can assume that $u = v^{1,0}$ or $u =
v^{5,0}$. The $\W$-module $M$ is spanned by the elements of the form
\eqref{eq:normal-form3}. Thus the weight $\lambda + 1$ subspace of $M$ is
spanned by $L(-1)u$, $W^p(-1)u$, $p = 3,4,5$. Let
\begin{equation*}
v = c_1 L(-1)u + c_2 W^3(-1)u + c_3 W^4(-1)u + c_4 W^5(-1)u
\end{equation*}
be an element of $M$ of weight $\lambda + 1$.

We first consider the case $u = v^{1,0}$. Then $L(0)$, $W^3(0)$, $W^4(0)$
and $W^5(0)$ act on $u$ as $5/96$, $20$, $780$ and $-1560$, respectively.
We study $L(1)v$ and $W^p(1)v$, $p = 3,4,5$. Each of these vectors is a
scalar multiple of $u$. Let $L(1)v = \eta_1 u$, $W^p(1)v = \eta_{p-1} u$,
$p = 3,4,5$. Using the expression of $W^i_n W^j$, $3 \le i \le j \le 5$, $0
\le n \le i + j -1$ as a linear combination of vectors of normal form given
in Appendix \ref{App:OPE}, together with basic formulas
\eqref{eq:basic-formula1} and \eqref{eq:basic-formula2}, we can determine
the constant $\eta_p$, $p = 1,2,3,4$. In fact, a suitable constant multiple
$F_p$ of $\eta_p$ is as follows.
\begin{align*}
F_1 &= c_1+576c_2+29952c_3-74880c_4,\\
F_2 &= 113c_1+65088c_2+3384576c_3+400721629860c_4,\\
F_3 &= 13c_1+7488c_2+13498935756c_3-973440c_4,\\
F_4 &=6217083815033c_1+169600079666412680418c_2\\
& \quad +186214094427868416c_3-62241257449122360326409060c_4.
\end{align*}

We can verify that a system of equations $F_1 = F_2 = F_3 = F_4 = 0$ has
only the trivial solution $c_1 = c_2 = c_3 = c_4 = 0$. That is, there is no
nonzero vector $v$ of weight $\lambda + 1$ in $M$ such that $L(1)v =0$ and
$W^p(1)v = 0$, $p = 3,4,5$. Since $\CE \cap (\lambda + \Z) = \{\lambda,
\lambda+1\}$, this implies that $M$ has no singular vector of weight
greater than $\lambda$. Hence $M$ is irreducible.

Next, we deal with the case $u = v^{5,0}$. The operators $L(0)$, $W^3(0)$,
$W^4(0)$ and $W^5(0)$ act on $u$ as $5/96$, $-20$, $780$ and $1560$,
respectively. We only need to replace $c_3$ with $-c_3$ and $c_5$ with
$-c_5$ in $F_1$, $F_2$, $F_3$ and $F_4$, and we obtain that $M$ is
irreducible.

Now, we can apply \cite[Proposition 3.11]{DLTYY} to conclude that the
vertex operator algebra $\W$ is rational. Thus the assertion (3) of
Conjecture \ref{Conj:W-properties} is true for $k = 6$.

\section*{Acknowledgments}
The authors would like to thank Toshiyuki Abe, Tomoyuki Arakawa, Atsushi
Matsuo and Hiroshi Yamauchi for helpful advice. Part of our computation was
done by a computer algebra system Risa/Asir. The authors are indebted to
Kazuhiro Yokoyama for a lot of advice concerning the system. H. Y. is
grateful to Lachezar S. Georgiev and Victor G. Kac for valuable discussions
on the subject. Chongying Dong was partially supported by NSF grants and a
research grant from University of California at Santa Cruz, Ching Hung Lam
was partially supported by  NSC grant 95-2115-M-006-013-MY2 of Taiwan,
Hiromichi Yamada was partially supported by JSPS Grant-in-Aid for
Scientific Research No. 20540012.

\appendix

\section{Virasoro primary vectors $W^3$, $W^4$ and $W^5$}
\label{App:def-W3-W4-W5}
{\footnotesize
\begin{equation*}
\begin{split}
W^3 &= k^2 h(-3)\1 + 3 k h(-2)h(-1)\1 +
2h(-1)^3\1 - 6k h(-1)e(-1)f(-1)\1 \\
& \quad + 3 k^2e(-2)f(-1)\1 - 3 k^2e(-1)f(-2)\1,
\end{split}
\end{equation*}

\begin{equation*}
\begin{split}
W^4 &= -2 k^2 (k^2+k+1) h(-4)\1 -8 k (k^2+k+1)h(-3)h(-1)\1 -k (5
k^2-6)h(-2)^2\1 \\
& \quad -2 k (11 k+6)h(-2)h(-1)^2\1 -(11 k+6)h(-1)^4\1 +4
k^2 (6 k-5)h(-2)e(-1)f(-1)\1 \\
& \quad +4 k (11 k+6)h(-1)^2e(-1)f(-1)\1 -4 k^2 (5
k+11)h(-1)e(-2)f(-1)\1 \\
& \quad +4 k^2 (5 k+11)h(-1)e(-1)f(-2)\1 +8 k^2 (k-3)
(k-2)e(-3)f(-1)\1 \\
& \quad -4 k^2 (3 k^2-3 k+8)e(-2)f(-2)\1 -2 k^2 (6
k-5)e(-1)^2f(-1)^2\1 \\
& \quad +8 k^2 (k^2+k+1)e(-1)f(-3)\1,
\end{split}
\end{equation*}

\begin{equation*}
\begin{split}
W^5 &= -2 k^3 (k^2+3 k+5)h(-5)\1 -10 k^2 (k^2+3 k+5)h(-4)h(-1)\1\\
& \quad -5 k^2 (3 k^2-4)h(-3)h(-2)\1 -5 k (7 k^2+12 k+16)h(-3)h(-1)^2\1\\
& \quad -15 k (3 k^2-4)h(-2)^2h(-1)\1 -5 k (19 k+12)h(-2)h(-1)^3\1 -2 (19
k+12)h(-1)^5\1 \\
& \quad +10 k^2 (4 k^2-7 k+8)h(-3)e(-1)f(-1)\1 +20 k^2 (10
k-7)h(-2)h(-1)e(-1)f(-1)\1 \\
& \quad +10 k (19 k+12)h(-1)^3e(-1)f(-1)\1 -5 k^2 (11
k^2-14 k+12)h(-2)e(-2)f(-1)\1 \\
& \quad -5 k^2 (17 k+64)h(-1)^2e(-2)f(-1)\1 +15 k^2 (3
k^2-4)h(-2)e(-1)f(-2)\1 \\
& \quad +5 k^2 (17 k+64)h(-1)^2e(-1)f(-2)\1 +30 k^2 (k-4)
(k-3)h(-1)e(-3)f(-1)\1 \\
& \quad -40 k^2 (k^2+3 k+5)h(-1)e(-2)f(-2)\1 -10 k^2 (10
k-7)h(-1)e(-1)^2f(-1)^2\1 \\
& \quad +10 k^2 (3 k^2+19 k+8)h(-1)e(-1)f(-3)\1 -10 k^3
(k-4) (k-3)e(-4)f(-1)\1 \\
& \quad +20 k^3 (k-4) (k-3)e(-3)f(-2)\1 +5 k^3 (10
k-7)e(-2)e(-1)f(-1)^2\1 \\
& \quad -10 k^3 (2 k^2-4 k+17)e(-2)f(-3)\1 -5 k^3 (10
k-7)e(-1)^2f(-2)f(-1)\1 \\
& \quad +10 k^3 (k^2+3 k+5)e(-1)f(-4)\1.
\end{split}
\end{equation*}
}

\section{$W^i_nW^j$, $3 \le i \le j \le 5$, $0 \le n \le i+j-1$}
\label{App:OPE}

\subsection{$W^3_n W^3$, $0 \le n \le 5$}
{\footnotesize
\begin{equation*}
W^3_5 W^3 = 12 k^3 (k-2) (k-1) (3 k+4) \1,
\end{equation*}
\begin{equation*}
W^3_4 W^3 = 0,
\end{equation*}
\begin{equation*}
W^3_3 W^3 = 36 k^3 (k-2) (k+2) (3 k+4) \omega_{-1}\1,
\end{equation*}
\begin{equation*}
W^3_2 W^3 = 18 k^3 (k-2) (k+2) (3 k+4) \omega_{-2}\1,
\end{equation*}
\begin{equation*}
\begin{split}
W^3_1 W^3 &=
-\big( 162 k^3 (k-2) (k+2) (3 k+4)/(16 k+17)\big) \omega_{-3}\1 \\
& \quad + \big(288 k^3 (k-2) (k+2)^2 (3 k+4)/(16 k+17)\big) \omega_{-1}\omega_{-1}\1\\
& \quad + \big(36 k (2 k+3)/(16 k+17)\big) W^4_{-1}\1,
\end{split}
\end{equation*}
\begin{equation*}
\begin{split}
W^3_0 W^3 &=
- \big(108 k^3 (k-2) (k+2) (3 k+4)/(16 k+17)\big) \omega_{-4}\1 \\
& \quad + \big(288 k^3 (k-2) (k+2)^2 (3 k+4)/(16 k+17)\big) \omega_{-2}\omega_{-1}\1\\
& \quad + \big(18 k (2 k+3)/(16 k+17)\big) W^4_{-2}\1.
\end{split}
\end{equation*}
}

\subsection{$W^3_n W^4$, $0 \le n \le 6$}
{\footnotesize
\begin{equation*}
W^3_6 W^4 = W^3_5 W^4 = W^3_4 W^4 = 0,
\end{equation*}
\begin{equation*}
W^3_3 W^4 = 48 k^2 (k-3) (2 k+1) (2 k+3) W^3_{-1}\1,
\end{equation*}
\begin{equation*}
W^3_2 W^4 = 16 k^2 (k-3) (2 k+1) (2 k+3) W^3_{-2}\1,
\end{equation*}
\begin{equation*}
\begin{split}
W^3_1 W^4 &=
\big(1248 k^2 (k-3) (k+2) (2 k+1) (2 k+3)/(64 k+107)\big) \omega_{-1}W^3_{-1}\1\\
& \quad - \big(48 k^2 (k-3) (2 k+1) (2 k+3) (2 k+7)/(64 k+107)\big) W^3_{-3}\1\\
& \quad - \big(12 k (3 k+4) (16 k+17)/(64 k+107)\big) W^5_{-1}\1,
\end{split}
\end{equation*}
\begin{equation*}
\begin{split}
W^3_0 W^4 &=
\big(120 k^2 (k-3) (k+2) (2 k+3) (16 k+17)/(64 k+107)\big) \omega_{-2}W^3_{-1}\1\\
& \quad + \big(48 k^2 (k-3) (k+2) (2 k+3) (8 k-11)/(64 k+107)\big) \omega_{-1}W^3_{-2}\1\\
& \quad - \big(12 k^2 (k-3) (2 k+3) (32 k^2+47 k-52)/(64 k+107)\big) W^3_{-4}\1\\
& \quad - \big(24 k (3 k+4) (16 k+17)/(5(64 k+107))\big) W^5_{-2}\1.
\end{split}
\end{equation*}
}

\subsection{$W^3_n W^5$, $0 \le n \le 7$}
{\footnotesize
\begin{equation*}
W^3_7 W^5 = W^3_6 W^5 = W^3_5 W^5 = W^3_4 W^5 = 0,
\end{equation*}
\begin{equation*}
W^3_3 W^5 = -30 k^2 (k-4) (5 k+8) W^4_{-1}\1,
\end{equation*}
\begin{equation*}
W^3_2 W^5  = -(15/2) k^2 (k-4) (5 k+8) W^4_{-2}\1,
\end{equation*}
\begin{equation*}
\begin{split}
W^3_1 W^5 &=
- \big(6480 k^4 (k+2) (k+3) (2 k+3) (3 k+4) (12 k^2+8 k-17)/(16 k+17)\big) \omega_{-5}\1\\
& \quad - \big(360 k^4 (k+2)^2 (2 k+3) (3 k+4) (32 k^2+797 k+863)/(16 k+17)\big) \omega_{-3}\omega_{-1}\1\\
& \quad + \big(45 k^4 (k+2)^2 (2 k+3) (3 k+4) (1408 k^2+1315 k-977)/(16 k+17)\big) \omega_{-2}\omega_{-2}\1\\
& \quad + \big(240 k^4 (k+2)^3 (2 k+3) (3 k+4) (202 k-169)/(16 k+17)\big) \omega_{-1}\omega_{-1}\omega_{-1}\1\\
& \quad - 15 k (2 k+3) (41 k+61) W^3_{-1}W^3_{-1}\1\\
& \quad + \big(60 k^2 (k+2) (404 k^2+1170 k+835)/(16 k+17)\big) \omega_{-1}W^4_{-1}\1\\
& \quad + \big(15 k^2 (2176 k^3+9481 k^2+13792 k+6708)/(2(16 k+17))\big)
W^4_{-3}\1,
\end{split}
\end{equation*}
\begin{equation*}
\begin{split}
W^3_0 W^5 &=
- \big(3240 k^4 (k+2) (k+3) (2 k+3) (3 k+4) (12 k^2+8 k-17)/(16 k+17)\big) \omega_{-6}\1\\
& \quad - \big(120 k^4 (k+2)^2 (2 k+3) (3 k+4) (184 k^2+1669 k+1801)/(16 k+17)\big) \omega_{-4}\omega_{-1}\1\\
& \quad + \big(2700 k^4 (k+2)^2 (2 k+3) (3 k+4) (8 k^2+5 k+5)/(16 k+17)\big) \omega_{-3}\omega_{-2}\1\\
& \quad + \big(240 k^4 (k+2)^3 (2 k+3) (3 k+4) (202 k-169)/(16 k+17)\big) \omega_{-2}\omega_{-1}\omega_{-1}\1\\
& \quad - 10 k (2 k+3) (41 k+61) W^3_{-2}W^3_{-1}\1\\
& \quad + \big(60 k^2 (k+2) (2 k+3) (64 k+107)/(16 k+17)\big) \omega_{-2}W^4_{-1}\1\\
& \quad + \big(60 k^2 (k+2) (138 k^2+382 k+257)/(16 k+17)\big) \omega_{-1}W^4_{-2}\1\\
& \quad + \big(15 k^2 (104 k^3+317 k^2+308 k+108)/(16 k+17)\big)
W^4_{-4}\1.
\end{split}
\end{equation*}
}

\subsection{$W^4_n W^4$, $0 \le n \le 7$}
{\footnotesize
\begin{equation*}
W^4_7 W^4 = 16 k^4 (k-3) (k-2) (k-1) (2 k+1) (3 k+4) (16 k+17) \1,
\end{equation*}
\begin{equation*}
W^4_6 W^4 = 0,
\end{equation*}

\begin{equation*}
W^4_5 W^4 = 64 k^4 (k-3) (k-2) (k+2) (2 k+1) (3 k+4) (16 k+17)
\omega_{-1}\1,
\end{equation*}

\begin{equation*}
W^4_4 W^4 = 32 k^4 (k-3) (k-2) (k+2) (2 k+1) (3 k+4) (16 k+17)
\omega_{-2}\1,
\end{equation*}

\begin{equation*}
\begin{split}
W^4_3 W^4 &=
-96 k^4 (k-3) (k-2) (k+2) (k+5) (2 k+1) (3 k+4) \omega_{-3}\1\\
& \quad +672 k^4 (k-3) (k-2) (k+2)^2 (2 k+1) (3 k+4) \omega_{-1}\omega_{-1}\1\\
& \quad +72 k^2 (4 k^3-15 k^2-33 k-4) W^4_{-1}\1,
\end{split}
\end{equation*}

\begin{equation*}
\begin{split}
W^4_2 W^4 &=
-64 k^4 (k-3) (k-2) (k+2) (k+5) (2 k+1) (3 k+4) \omega_{-4}\1\\
& \quad +672 k^4 (k-3) (k-2) (k+2)^2 (2 k+1) (3 k+4) \omega_{-2}\omega_{-1}\1\\
& \quad +36 k^2 (4 k^3-15 k^2-33 k-4) W^4_{-2}\1,
\end{split}
\end{equation*}

\begin{equation*}
\begin{split}
W^4_1 W^4 &=
1920 k^4 (k+2) (2 k+3) (3 k+4) (4 k^3+12 k^2-4 k-9) \omega_{-5}\1\\
& \quad +32 k^4 (k+2)^2 (3 k+4) (62 k^3+1869 k^2+3337 k+1524) \omega_{-3}\omega_{-1}\1\\
& \quad -40 k^4 (k+2)^2 (3 k+4) (316 k^3+534 k^2+29 k-165) \omega_{-2}\omega_{-2}\1\\
& \quad -320 k^4 (k+2)^3 (3 k+4) (5 k+4) (6 k-5) \omega_{-1}\omega_{-1}\omega_{-1}\1\\
& \quad +8 k (k+1) (16 k+17)^2 W^3_{-1}W^3_{-1}\1\\
& \quad -48 k^2 (k+2) (52 k^2+109 k+50) \omega_{-1}W^4_{-1}\1\\
& \quad -4 k^2 (412 k^3+1503 k^2+1753 k+640) W^4_{-3}\1,
\end{split}
\end{equation*}

\begin{equation*}
\begin{split}
W^4_0 W^4 &=
1440 k^4 (k+2) (2 k+3) (3 k+4) (4 k^3+12 k^2-4 k-9) \omega_{-6}\1\\
& \quad +96 k^4 (k+2)^2 (3 k+4) (11 k+13) (6 k^2+55 k+41) \omega_{-4}\omega_{-1}\1\\
& \quad -48 k^4 (k+2)^2 (3 k+4) (136 k^3+180 k^2+161 k+75) \omega_{-3}\omega_{-2}\1\\
& \quad -480 k^4 (k+2)^3 (3 k+4) (5 k+4) (6 k-5) \omega_{-2}\omega_{-1}\omega_{-1}\1\\
& \quad +8 k (k+1) (16 k+17)^2 W^3_{-2}W^3_{-1}\1\\
& \quad -24 k^2 (k+2) (52 k^2+109 k+50) \omega_{-2}W^4_{-1}\1\\
& \quad -24 k^2 (k+2) (52 k^2+109 k+50) \omega_{-1}W^4_{-2}\1\\
& \quad -12 k^2 (20 k^3+45 k^2+29 k+8) W^4_{-4}\1.
\end{split}
\end{equation*}
}

\subsection{$W^4_n W^5$, $0 \le n \le 8$}
{\footnotesize
\begin{equation*}
W^4_8 W^5 = W^4_7 W^5 = W^4_6 W^5 = 0,
\end{equation*}
\begin{equation*}
W^4_5 W^5 = -40 k^3 (k-4) (k-3) (2 k+1) (5 k+8) (16 k+17) W^3_{-1}\1,
\end{equation*}
\begin{equation*}
W^4_4 W^5 = -(40/3) k^3 (k-4) (k-3) (2 k+1) (5 k+8) (16 k+17) W^3_{-2}\1,
\end{equation*}
\begin{equation*}
\begin{split}
W^4_3 W^5 & =
- \big(1320 k^3 (k-4) (k-3) (k+2) (2 k+1) (5 k+8) (16 k+17)/(64 k+107)\big) \omega_{-1}W^3_{-1}\1\\
& \quad + \big(40 k^3 (k-4) (k-3) (2 k+1) (5 k+8) (5 k+13) (16 k+17)/(64 k+107)\big) W^3_{-3}\1\\
& \quad + \big(180 k^2 (3 k+4) (32 k^3-236 k^2-535 k-125)/(64 k+107)\big)
W^5_{-1}\1,
\end{split}
\end{equation*}
\begin{equation*}
\begin{split}
W^4_2 W^5 &=
- \big(160 k^3 (k-4) (k-3) (k+2) (5 k+8) (13 k+14) (16 k+17)/(64 k+107)\big) \omega_{-2}W^3_{-1}\1\\
& \quad - \big( (80/3) k^3 (k-4) (k-3) (k+2) (5 k+8) (14 k-23) (16 k+17)/(64 k+107)\big) \omega_{-1}W^3_{-2}\1\\
& \quad + \big(60 k^3 (k-4) (k-3) (5 k+8) (16 k+17) (8 k^2+12 k-11)/(64 k+107)\big) W^3_{-4}\1\\
& \quad + \big(72 k^2 (3 k+4) (32 k^3-236 k^2-535 k-125)/(64 k+107)\big)
W^5_{-2}\1,
\end{split}
\end{equation*}
\begin{equation*}
\begin{split}
W^4_1 W^5 &=
\big(20 k^3 (k+2) (5 k+8) (2624 k^4-83108 k^3-341706 k^2-433511 k -177319)/(64 k+107)\big) \omega_{-3}W^3_{-1}\1\\
& \quad + \big(120 k^3 (k+2)^2 (2 k+1) (5 k+8) (16 k-9) (75 k+74)/(64 k+107)\big) \omega_{-1}\omega_{-1}W^3_{-1}\1\\
& \quad + \big((10/3) k^3 (k+2) (5 k+8) (16 k+17) (7960 k^3+18296 k^2+6119 k -4457)/(64 k+107)\big) \omega_{-2}W^3_{-2}\1\\
& \quad - \big((40/3) k^3 (k+2) (5 k+8) (28800 k^4+128704 k^3+133404 k^2-43341 k -76171)\\
& \qquad\qquad /(64 k+107)\big) \omega_{-1}W^3_{-3}\1\\
& \quad + \big((20/3) k^3 (5 k+8) (45440 k^5+358008 k^4+884944 k^3+619369 k^2 -360351 k-404824)\\
& \qquad\qquad /(64 k+107)\big) W^3_{-5}\1\\
& \quad -10 k (5 k+8) (20 k+19) W^3_{-1}W^4_{-1}\1\\
& \quad - \big(40 k^2 (k+2) (3 k+4) (1168 k^2+2584 k+1171)/(64 k+107)\big) \omega_{-1}W^5_{-1}\1\\
& \quad - \big(120 k^2 (2 k+3) (3 k+4) (88 k^2+202 k+97)/(64 k+107)\big)
W^5_{-3}\1,
\end{split}
\end{equation*}

\begin{equation*}
\begin{split}
W^4_0 W^5 &=
- \big(60 k^3 (k+2) (5 k+8) (16 k+17) (236 k^3+2566 k^2+5577 k+3207)/(64 k+107)\big) \omega_{-4}W^3_{-1}\1\\
& \quad + \big(2640 k^3 (k+1) (k+2)^2 (5 k+8) (10 k-7) (16 k+17)/(64 k+107)\big) \omega_{-2}\omega_{-1}W^3_{-1}\1\\
& \quad + \big(40 k^3 (k+2) (5 k+8) (6976 k^4+26048 k^3+28428 k^2+2426 k -6881)/(64 k+107)\big) \omega_{-3}W^3_{-2}\1\\
& \quad + \big(40 k^3 (k+2)^2 (5 k+8) (160 k^3-2938 k^2+215 k+3238)/(64 k+107)\big) \omega_{-1}\omega_{-1}W^3_{-2}\1\\
& \quad - \big(20 k^3 (k+2) (5 k+8) (16 k+17) (60 k^3+202 k^2-305 k-407)/(64 k+107)\big) \omega_{-2}W^3_{-3}\1\\
& \quad - \big(120 k^3 (k+2) (5 k+8) (16 k+17) (150 k^3+407 k^2+93 k-149)/(64 k+107)\big) \omega_{-1}W^3_{-4}\1\\
& \quad + \big(20 k^3 (5 k+8) (26368 k^5+171408 k^4+368012 k^3+239153 k^2-101871 k -117440)\\
& \qquad\qquad /(64 k+107)\big) W^3_{-6}\1\\
& \quad -10 k (5 k+8) (20 k+19) W^3_{-2}W^4_{-1}\1\\
& \quad - \big(240 k^2 (k+2) (2 k+3) (180 k^2+397 k+179)/(64 k+107)\big) \omega_{-2}W^5_{-1}\1\\
& \quad - \big(24 k^2 (k+2) (2064 k^3+7088 k^2+7653 k+2536)/(64 k+107)\big) \omega_{-1}W^5_{-2}\1\\
& \quad - \big(12 k^2 (16 k+17) (36 k^3+107 k^2+107 k+44)/(64 k+107)\big)
W^5_{-4}\1.
\end{split}
\end{equation*}
}

\subsection{$W^5_n W^5$, $0 \le n \le 9$}
{\footnotesize
\begin{equation*}
W^5_9 W^5 = 40 k^5 (k-4) (k-3) (k-2) (k-1) (2 k+1) (5 k+8) (64 k+107)\1,
\end{equation*}
\begin{equation*}
W^5_8 W^5 = 0,
\end{equation*}
\begin{equation*}
W^5_7 W^5 = 200 k^5 (k-4) (k-3) (k-2) (k+2) (2 k+1) (5 k+8) (64 k+107)
\omega_{-1}\1,
\end{equation*}
\begin{equation*}
W^5_6 W^5 = 100 k^5 (k-4) (k-3) (k-2) (k+2) (2 k+1) (5 k+8) (64 k+107)
\omega_{-2}\1,
\end{equation*}
\begin{equation*}
\begin{split}
W^5_5 W^5 &=
- \big(300 k^5 (k-4) (k-3) (k-2) (k+2) (2 k+1) (2 k+7) (5 k+8) (64 k+107) /(16 k+17)\big) \omega_{-3}\1\\
& \quad + \big(2600 k^5 (k-4) (k-3) (k-2) (k+2)^2 (2 k+1) (5 k+8) (64 k+107)/(16 k+17)\big) \omega_{-1}\omega_{-1}\1\\
& \quad + \big(450 k^3 (k-4) (5 k+8) (32 k^3-236 k^2-535 k-125)/(16
k+17)\big) W^4_{-1}\1,
\end{split}
\end{equation*}

\begin{equation*}
\begin{split}
W^5_4 W^5 &=
- \big(200 k^5 (k-4) (k-3) (k-2) (k+2) (2 k+1) (2 k+7) (5 k+8) (64 k+107)/(16 k+17)\big) \omega_{-4}\1\\
& \quad + \big(2600 k^5 (k-4) (k-3) (k-2) (k+2)^2 (2 k+1) (5 k+8) (64 k+107)/(16 k+17)\big) \omega_{-2}\omega_{-1}\1\\
& \quad + \big(225 k^3 (k-4) (5 k+8) (32 k^3-236 k^2-535 k-125)/(16
k+17)\big) W^4_{-2}\1,
\end{split}
\end{equation*}

\begin{equation*}
\begin{split}
W^5_3 W^5 &=
\big(400 k^5 (k+2) (9728 k^7-345370 k^6-2884229 k^5-7339690 k^4-5652707 k^3 +3682145 k^2\\
& \qquad\qquad +6580220 k+1862400)/(16 k+17)\big) \omega_{-5}\1\\
& \quad - \big(600 k^5 (k+2)^2 (256 k^6+7900 k^5+1054568 k^4+4865734 k^3+8044197 k^2 \\
& \qquad\qquad +5415116 k+1171136)/(16 k+17)\big) \omega_{-3}\omega_{-1}\1\\
& \quad - \big(25 k^5 (k+2)^2 (137728 k^6-4923496 k^5-24095252 k^4-35522641 k^3 -12391265 k^2\\
& \qquad\qquad +8406500 k+3657600)/(16 k+17)\big) \omega_{-2}\omega_{-2}\1\\
& \quad - \big(200 k^5 (k+2)^3 (12288 k^5-487882 k^4-1447853 k^3-491400 k^2+1135840 k +463040)\\
& \qquad\qquad /(16 k+17)\big) \omega_{-1}\omega_{-1}\omega_{-1}\1\\
& \quad + 25 k^2 (544 k^4-16660 k^3-65657 k^2-72453 k-19600) W^3_{-1}W^3_{-1}\1\\
& \quad - \big(150 k^3 (k+2) (1632 k^4-54468 k^3-209305 k^2-225706 k -59200)/(16 k+17)\big) \omega_{-1}W^4_{-1}\1\\
& \quad - \big(25 k^3 (6816 k^5-206652 k^4-1172123 k^3-2196873 k^2-1637466 k -371200)/(16 k+17)\big) W^4_{-3}\1,
\end{split}
\end{equation*}

\begin{equation*}
\begin{split}
W^5_2 W^5 &=
\big(300 k^5 (k+2) (9728 k^7-345370 k^6-2884229 k^5-7339690 k^4-5652707 k^3 +3682145 k^2\\
& \qquad\qquad +6580220 k+1862400)/(16 k+17)\big) \omega_{-6}\1\\
& \quad + \big(100 k^5 (k+2)^2 (11264 k^6-476132 k^5-8238118 k^4-32234405 k^3 -50927083 k^2\\
& \qquad\qquad -34091060 k-7406592)/(16 k+17)\big) \omega_{-4}\omega_{-1}\1\\
& \quad - \big(150 k^5 (k+2)^2 (12800 k^6-428732 k^5-1910710 k^4-3040001 k^3-2661901 k^2 \\
& \qquad\qquad -1600364 k-379776)/(16 k+17)\big) \omega_{-3}\omega_{-2}\1\\
& \quad - \big(300 k^5 (k+2)^3 (12288 k^5-487882 k^4-1447853 k^3-491400 k^2+1135840 k +463040)\\
& \qquad\qquad /(16 k+17)\big) \omega_{-2}\omega_{-1}\omega_{-1}\1\\
& \quad + 25 k^2 (544 k^4-16660 k^3-65657 k^2-72453 k-19600) W^3_{-2}W^3_{-1}\1\\
& \quad - \big(75 k^3 (k+2) (1632 k^4-54468 k^3-209305 k^2-225706 k -59200)/(16 k+17)\big) \omega_{-2}W^4_{-1}\1\\
& \quad - \big(75 k^3 (k+2) (1632 k^4-54468 k^3-209305 k^2-225706 k -59200)/(16 k+17)\big) \omega_{-1}W^4_{-2}\1\\
& \quad - \big(75 k^3 (384 k^5-10608 k^4-43480 k^3-52785 k^2-21126
k-3200)/(16 k+17)\big) W^4_{-4}\1,
\end{split}
\end{equation*}
}

\newpage
{\footnotesize
\begin{equation*}
\begin{split}
W^5_1 W^5 &=
- \big(25 k^5 (k+2) (170027057152 k^{10}+2356580095488 k^9+13676829114720 k^8 \\
& \qquad\qquad +42735238046312 k^7+74793658083474 k^6+60828640020771 k^5 \\
& \qquad\qquad -18630034236787 k^4-94115224713312 k^3-92379208458276 k^2 \\
& \qquad\qquad -41524883935184 k-7347787324608)/(17 (k+1) (16 k+17)^2 (64 k+107))\big) \omega_{-7}\1\\
& \quad - \big(600 k^5 (k+2)^2 (3950979072 k^9+64351956480 k^8+426934964416 k^7 +1559551526014 k^6\\
& \qquad\qquad +3508072702889 k^5+5080429991599 k^4 +4758480876095 k^3+2785329492007 k^2 \\
& \qquad\qquad +924445415740 k+132238676112)/(17 (k+1) (16 k+17)^2 (64 k+107))\big) \omega_{-5}\omega_{-1}\1\\
& \quad + \big(150 k^5 (k+2)^2 (19297443840 k^9+207821377280 k^8+951110388112 k^7+2382556615296 k^6 \\
& \qquad\qquad +3466010214237 k^5+2750850253947 k^4 +729192953742k^3-542736421532 k^2 \\
& \qquad\qquad -472733750908 k-105657243168)/(17 (k+1) (16 k+17)^2 (64 k+107))\big) \omega_{-4}\omega_{-2}\1\\
& \quad - \big(600 k^5 (k+2)^2 (9399296 k^9+115780640 k^8+833628608 k^7+3394583580 k^6 \\
& \qquad\qquad +6845449350 k^5+3538838811 k^4-10856512820 k^3-21519656045 k^2 \\
& \qquad\qquad -14915524340 k-3556147692)/(17 (k+1) (16 k+17)^2 (64 k+107))\big) \omega_{-3}\omega_{-3}\1\\
& \quad - \big(100 k^5 (k+2)^3 (49197952 k^7+1316567424 k^6+7996912590 k^5+22351947279 k^4 +34102766376 k^3\\
& \qquad\qquad +29430436515 k^2+13488718172 k +2525554272)/(17 (k+1) (16 k+17)^2)\big) \omega_{-3}\omega_{-1}\omega_{-1}\1\\
& \quad + \big((25/2) k^5 (k+2)^3 (1808895488 k^7+12570420576 k^6+34774013520 k^5 \\
& \qquad\qquad +47177699046 k^4+28725570387 k^3+203645169 k^2-8060040140 k \\
& \qquad\qquad -2646690336)/(17 (k+1) (16 k+17)^2)\big) \omega_{-2}\omega_{-2}\omega_{-1}\1\\
& \quad + \big(200 k^5 (k+2)^4 (87914016 k^6+458469088 k^5+817692930 k^4+409776935 k^3 \\
& \qquad\qquad -394611175 k^2-507021564 k-148662512)/(17 (k+1) (16 k+17)^2)\big) \omega_{-1}\omega_{-1}\omega_{-1}\omega_{-1}\1\\
& \quad - \big((25/2) k^2 (k+2) (23714656 k^5+162824160 k^4+438179214 k^3+573852691 k^2 \\
& \qquad\qquad +361829501 k+86225720)/(17 (k+1) (64 k+107))\big) \omega_{-1}W^3_{-1}W^3_{-1}\1\\
& \quad + \big(25 k^2 (2088128 k^6+10843456 k^5+10746960 k^4-34653451 k^3-92307847 k^2 \\
& \qquad\qquad -77110884 k-20771320)/(17 (k+1) (64 k+107))\big) W^3_{-3}W^3_{-1}\1\\
& \quad - \big(300 k^3 (k+2) (504864 k^6+454764 k^5-10972571 k^4-41935547 k^3-63119109 k^2 \\
& \qquad\qquad -42865453 k-10561420)/(17 (k+1) (16 k+17)^2)\big) \omega_{-3}W^4_{-1}\1\\
& \quad + \big(150 k^3 (k+2)^2 (9637952 k^5+58430080 k^4+139113500 k^3+162223837 k^2 \\
& \qquad\qquad +92206985 k+20184520)/(17 (k+1) (16 k+17)^2)\big) \omega_{-1}\omega_{-1}W^4_{-1}\1\\
& \quad - \big((75/8) k^3 (k+2) (10958592 k^6+71925552 k^5+180045944 k^4+196657979 k^3 \\
& \qquad\qquad +56586147 k^2-47214608 k-24536960)/(17 (k+1) (16 k+17)^2)
\big) \omega_{-2}W^4_{-2}\1\\
& \quad + \big((25/2) k^3 (k+2) (81593088 k^6+609224400 k^5+1877915632 k^4+3065560099 k^3 \\
& \qquad\qquad +2797941765 k^2+1347940946 k+262697360)/(17 (k+1) (16 k+17)^2)\big) \omega_{-1}W^4_{-3}\1\\
& \quad - \big(150 k^3 (18189312 k^8+131109696 k^7+275959096 k^6-155647282 k^5 -1459360618 k^4-2182319517 k^3\\
& \qquad\qquad -1238636927 k^2-142994954 k +43729880)/(17 (k+1) (16 k+17)^2 (64 k+107))\big) W^4_{-5}\1\\
& \quad - \big(75 k (21712 k^4+134672 k^3+295445 k^2+266275 k +78990)/(17 (k+1) (16 k+17)^2)\big) W^4_{-1}W^4_{-1}\1\\
& \quad - \big((75/2) k (47824 k^4+265368 k^3+534533 k^2+455349 k+133560)/(17 (k+1) (64 k+107))\big) W^3_{-1}W^5_{-1}\1,
\end{split}
\end{equation*}
}

{\footnotesize
\begin{equation*}
\begin{split}
W^5_0 W^5 &=
- \big(200 k^5 (k+2) (5244455936 k^{10}+70196859904 k^9+398153246900 k^8 \\
& \qquad\qquad +1233587976582 k^7+2201016521153 k^6+2027732203871 k^5 \\
& \qquad\qquad +165492882072 k^4-1736353418536 k^3-1889476600504 k^2 \\
& \qquad\qquad -869314211744 k-154555776032)/((16 k+17)^2 (131 k^2+351 k+229))\big) \omega_{-8}\1\\
& \quad - \big(400 k^5 (k+2)^2 (724096768 k^9+13032033712 k^8+92438438140 k^7+355708801440 k^6 \\
& \qquad\qquad +835873623678 k^5+1257476731455 k^4+1218124586462 k^3+734755072085 k^2 \\
& \qquad\qquad +250569685256 k+36757401408)/((16 k+17)^2 (131 k^2+351 k+229))\big) \omega_{-6}\omega_{-1}\1\\
& \quad - \big(400 k^5 (k+2)^2 (214796800 k^9+2717542336 k^8+13231630738 k^7+30829431549 k^6 \\
& \qquad\qquad +28628760321 k^5-20502956550 k^4-80533588567 k^3-85066432663 k^2 \\
& \qquad\qquad -41491861540 k-7925033520)/((16 k+17)^2 (131 k^2+351 k+229))\big) \omega_{-5}\omega_{-2}\1\\
& \quad + \big(200 k^5 (k+2)^2 (268100608 k^9+2728543168 k^8+12034185208 k^7+30531316020 k^6 \\
& \qquad\qquad +50376301665 k^5+58543428453 k^4+50148887738 k^3+30806117372 k^2 \\
& \qquad\qquad +11758065152 k+1966237968)/((16 k+17)^2 (131 k^2+351 k+229))\big) \omega_{-4}\omega_{-3}\1\\
& \quad + \big(1200 k^5 (k+2)^3 (57147904 k^8+134150272 k^7-1594361818 k^6-10000010223 k^5 \\
& \qquad\qquad -25615972350 k^4-35776722309 k^3-28439133457 k^2-12042263860 k \\
& \qquad\qquad -2090245832)/((16 k+17)^2 (131 k^2+351 k+229))\big) \omega_{-4}\omega_{-1}\omega_{-1}\1\\
& \quad + \big(200 k^5 (k+2)^3 (101861888 k^8-2776465376 k^7-28779796440 k^6-112345101336 k^5 \\
& \qquad\qquad -234716374893 k^4-287794828491 k^3-207844230254 k^2-81832486600 k \\
& \qquad\qquad -13505087256)/((16 k+17)^2 (131 k^2+351 k+229))\big) \omega_{-3}\omega_{-2}\omega_{-1}\1\\
& \quad + \big(25 k^5 (k+2)^3 (410452736 k^7+3102590048 k^6+9337355932 k^5 \\
& \qquad\qquad +13766811166 k^4+9179515875 k^3+441625691 k^2-2501518260 k \\
& \qquad\qquad -881437904)/((16 k+17) (131 k^2+351 k+229))\big) \omega_{-2}\omega_{-2}\omega_{-2}\1\\
& \quad + \big(400 k^5 (k+2)^4 (632017152 k^7+4297675600 k^6+11097398636 k^5 \\
& \qquad\qquad +12238161284 k^4+1797240891 k^3-8140638059 k^2-6817742300 k \\
& \qquad\qquad -1683752944)/((16 k+17)^2 (131 k^2+351 k+229))\big) \omega_{-2}\omega_{-1}\omega_{-1}\omega_{-1}\1\\
& \quad - \big(25 k^2 (k+2) (1346512 k^5+9136020 k^4+24342198 k^3+31659247 k^2 \\
& \qquad\qquad +19927547 k+4787240)/(131 k^2+351 k+229)\big) \omega_{-2}W^3_{-1}W^3_{-1}\1\\
& \quad - \big(100 k^2 (k+2) (227056 k^5+1523040 k^4+3994369 k^3+5084081 k^2 \\
& \qquad\qquad +3107416 k+716920)/(131 k^2+351 k+229)\big) \omega_{-1}W^3_{-2}W^3_{-1}\1\\
& \quad + \big(25 k^2 (218864 k^6+1387084 k^5+3127878 k^4+2581115 k^3-382393 k^2 \\
& \qquad\qquad -1483740 k-448840)/(131 k^2+351 k+229)\big) W^3_{-4}W^3_{-1}\1\\
& \quad + \big(300 k^3 (k+2) (13304576 k^7+146690032 k^6+668880716 k^5 +1643048125 k^4+2351854427 k^3\\
& \qquad\qquad +1958853655 k^2+874329498 k +159392280)/((16 k+17)^2 (131 k^2+351 k+229)) \big) \omega_{-4}W^4_{-1}\1\\
& \quad + \big(300 k^3 (k+2)^2 (44903552 k^6+347191688 k^5+1102896188 k^4 +1838425452 k^3+1691025439 k^2\\
& \qquad\qquad +810350893 k +156998080)/((16 k+17)^2 (131 k^2+351 k+229))\big) \omega_{-2}\omega_{-1}W^4_{-1}\1\\
& \quad - \big(75 k^3 (k+2) (12201216 k^7+95826768 k^6+305102056 k^5+494975652 k^4 +412040321 k^3\\
& \qquad\qquad +138210893 k^2-14086422 k -13469960)/((16 k+17)^2 (131 k^2+351 k+229))\big) \omega_{-3}W^4_{-2}\1
\end{split}
\end{equation*}
}

\newpage
{\footnotesize
\begin{equation*}
\begin{split}
\hspace*{1cm}
& \quad + \big(300 k^3 (k+2)^2 (12518976 k^6+93551364 k^5+286503394 k^4 +459301951 k^3+405337557 k^2\\
& \qquad\qquad +185853764 k +34297340)/((16 k+17)^2 (131 k^2+351 k+229))\big) \omega_{-1}\omega_{-1}W^4_{-2}\1\\
& \quad + \big(50 k^3 (k+2) (8988192 k^6+73184988 k^5+245480957 k^4 +433108373 k^3+422458956 k^2\\
& \qquad\qquad +214916428 k +44202700)/((16 k+17) (131 k^2+351 k+229))\big) \omega_{-2}W^4_{-3}\1\\
& \quad + \big(150 k^3 (k+2) (2 k+3) (2310272 k^6+14302664 k^5+35897052 k^4 +47185363 k^3+34385905 k^2\\
& \qquad\qquad +12495594 k+1193880)/((16 k+17)^2 (131 k^2+351 k+229))\big) \omega_{-1}W^4_{-4}\1\\
& \quad - \big(50 k^3 (3095808 k^8+22633488 k^7+54967220 k^6+17421355 k^5 -157516472 k^4-327442197 k^3\\
& \qquad\qquad -312185014 k^2-170931184 k -48119360)/((16 k+17)^2 (131 k^2+351 k+229))\big) W^4_{-6}\1\\
& \quad - \big(75 k (2 k+3) (64 k+107) (368 k^3+1897 k^2+2671 k +950)\\
& \qquad\qquad /((16 k+17)^2 (131 k^2+351 k+229))\big) W^4_{-2}W^4_{-1}\1\\
& \quad - \big(50 k (2 k+3) (5 k+8) (388 k^2+863 k+405)/(131 k^2+351 k+229)
\big) W^3_{-2}W^5_{-1}\1.
\end{split}
\end{equation*}
}

\newpage
\section{Linear relations in $\tW$}\label{App:linear-relations}
\subsection{Linear relations in the weight $8$ subspace}\label{App:wt8-linear-relations}
We express $(W^3_{-2})^2\1$ and $W^3_{-1}W^4_{-2}\1$ as linear combinations
of the remaining $27$ vectors of normal form of weight $8$.
{\footnotesize
\begin{equation*}
\begin{split}
(W^3_{-2})^2\1 &=
- \big(18 k^3 (k+2) (3 k+4) (217088 k^5+1323552 k^4+1864570 k^3-459533 k^2 \\
& \qquad\qquad -1453848 k-18520)/(17 (k+1) (16 k+17)^2 (64 k+107))\big) \omega_{-7}\1\\
& \quad - \big(288 k^3 (k+2)^2 (3 k+4) (6976 k^4+112048 k^3+316803 k^2 \\
& \qquad\qquad +301883 k+91892)/(17 (k+1) (16 k+17)^2 (64 k+107))\big) \omega_{-5}\omega_{-1}\1\\
& \quad + \big(54 k^3 (k+2)^2 (3 k+4) (56320 k^4-6240 k^3-436698 k^2 \\
& \qquad\qquad -541975 k-173763)/(17 (k+1) (16 k+17)^2 (64 k+107))\big) \omega_{-4}\omega_{-2}\1\\
& \quad + \big(972 k^3 (k+2)^2 (3 k+4) (1792 k^4+5096 k^3+6100 k^2 \\
& \qquad\qquad +4783 k+229)/(17 (k+1) (16 k+17)^2 (64 k+107))\big) (\omega_{-3})^2\1\\
& \quad + \big(72 k^3 (k+2)^3 (3 k+4) (920 k^2+198 k +187)/(17 (k+1) (16 k+17)^2)\big) \omega_{-3}(\omega_{-1})^2\1\\
& \quad + \big(9 k^3 (k+2)^3 (3 k+4) (5792 k^2-1566 k -3425)/(17 (k+1) (16 k+17)^2)\big) (\omega_{-2})^2\omega_{-1}\1\\
& \quad - \big(1008 k^3 (k+2)^4 (3 k+4) (6 k-5)/(17 (k+1) (16 k+17)^2)\big) (\omega_{-1})^4\1\\
& \quad + \big(9 (k+2) (26 k+83)/(17 (k+1) (64 k+107))\big) \omega_{-1}(W^3_{-1})^2\1\\
& \quad - \big(6 (380 k^2+822 k+301)/(17 (k+1) (64 k+107))\big) W^3_{-3}W^3_{-1}\1\\
& \quad + \big(54 k (k+2) (120 k^2+141 k-34)/(17 (k+1) (16 k+17)^2)\big) \omega_{-3}W^4_{-1}\1\\
& \quad - \big(36 k (k+2)^2 (36 k+61)/(17 (k+1) (16 k+17)^2)\big) (\omega_{-1})^2W^4_{-1}\1\\
& \quad + \big(27 k (k+2) (784 k^2+1565 k+664)/(68 (k+1) (16 k+17)^2)\big) \omega_{-2}W^4_{-2}\1\\
& \quad + \big(9 k (k+2) (496 k^2+889 k+214)/(17 (k+1) (16 k+17)^2)\big) \omega_{-1}W^4_{-3}\1\\
& \quad + \big(54 k (2688 k^4+9208 k^3+9951 k^2+5793 k \\
& \qquad\qquad +3788)/(17 (k+1) (16 k+17)^2 (64 k+107))\big) W^4_{-5}\1\\
& \quad + \big(18/(17 k (k+1) (16 k+17)^2)\big) (W^4_{-1})^2\1\\
& \quad + \big(9/(17 k (k+1) (64 k+107))\big) W^3_{-1}W^5_{-1}\1,
\end{split}
\end{equation*}

\begin{equation*}
\begin{split}
W^3_{-1}W^4_{-2}\1 &=
\big(8 k^2 (k+2) (2 k+3) (320 k^2-155 k-621)/(64 k+107)\big) \omega_{-4}W^3_{-1}\1\\
& \quad - \big(8 k^2 (k+2)^2 (2 k+3) (136 k-109)/(64 k+107)\big) \omega_{-2}\omega_{-1}W^3_{-1}\1\\
& \quad + \big(16 k^2 (k+2) (2 k+3) (40 k^2-310 k-327)/(64 k+107)\big) \omega_{-3}W^3_{-2}\1\\
& \quad + \big(16 k^2 (k+2)^2 (2 k+3) (136 k-109)/(3 (64 k+107))\big) (\omega_{-1})^2W^3_{-2}\1\\
& \quad + \big(8 k^2 (k+2) (2 k+3) (168 k^2-16 k+31)/(64 k+107)\big) \omega_{-2}W^3_{-3}\1\\
& \quad - \big(20 k^2 (k+2) (2 k+3) (64 k^2+167 k-180)/(64 k+107)\big) \omega_{-1}W^3_{-4}\1\\
& \quad + \big(24 k^2 (2 k+3) (32 k^3+344 k^2+347 k-516)/(64 k+107)\big) W^3_{-6}\1\\
& \quad + (4/3) W^3_{-2}W^4_{-1}\1\\
& \quad + \big(4 k (k+2) (16 k+17)/(64 k+107)\big) \omega_{-2}W^5_{-1}\1\\
& \quad - \big(8 k (k+2) (16 k+17)/(5 (64 k+107))\big) \omega_{-1}W^5_{-2}\1\\
& \quad - \big(8 k (7 k+9) (16 k+17)/(5 (64 k+107))\big) W^5_{-4}\1,
\end{split}
\end{equation*}
}

\subsection{Linear relations in the weight $9$ subspace}\label{App:wt9-linear-relations}
We can express $W^3_{-3}W^3_{-2}\1$, $W^3_{-2}W^4_{-2}\1$,
$W^3_{-1}W^4_{-3}\1$ and $W^3_{-1}W^5_{-2}\1$ as linear combinations of the
remaining $40$ vectors of normal form of weight $9$. For instance,
$W^3_{-1}W^4_{-3}\1$ is expressed as follows. We omit the expression of
$W^3_{-3}W^3_{-2}\1$, $W^3_{-2}W^4_{-2}\1$ and $W^3_{-1}W^5_{-2}\1$, for
they are not used in our argument. {\footnotesize
\begin{equation*}
\begin{split}
W^3_{-1}W^4_{-3}\1 &=
\big(16 k^2 (k+2) (1283648 k^5+3440448 k^4-3245504 k^3-18095627 k^2-18583431 k-5789692)\\
& \qquad\qquad /((64 k+107) (1424 k^2+3241 k+1542)) \omega_{-5}W^3_{-1}\1\\
& \quad - \big(8 k^2 (k+2)^2 (455808 k^4+2157980 k^3+3327583 k^2+1752535 k+133700)\\
& \qquad\qquad /((64 k+107) (1424 k^2+3241 k+1542)) \omega_{-3}\omega_{-1}W^3_{-1}\1\\
& \quad - \big(k^2 (k+2)^2 (8192 k^3-432 k^2-30515 k-15420)/(1424 k^2+3241 k+1542) (\omega_{-2})^2W^3_{-1}\1\\
& \quad + \big(16 k^2 (k+2)^3 (674 k^2+637 k-1100)/(1424 k^2+3241 k+1542) (\omega_{-1})^3W^3_{-1}\1\\
& \quad + \big(8 k^2 (k+2) (5 k+8) (16 k+17) (58944 k^3-96692 k^2-505205 k-340649)\\
& \qquad\qquad /(3 (64 k+107) (1424 k^2+3241 k+1542)) \omega_{-4}W^3_{-2}\1\\
& \quad + \big(4 k^2 (k+2)^2 (5 k+8) (16 k+17) (11248 k^2-3953 k-6251)\\
& \qquad\qquad /(3 (64 k+107) (1424 k^2+3241 k+1542)) \omega_{-2}\omega_{-1}W^3_{-2}\1\\
& \quad + \big(8 k^2 (k+2) (1209472 k^5+1405772 k^4-12112961 k^3-34155325 k^2-32424710 k-10585768)\\
& \qquad\qquad /((64 k+107) (1424 k^2+3241 k+1542)) \omega_{-3}W^3_{-3}\1\\
& \quad + \big(32 k^2 (k+2)^2 (627872 k^4+2346827 k^3+2091732 k^2-1239437 k-1843412)\\
& \qquad\qquad /(3 (64 k+107) (1424 k^2+3241 k+1542)) (\omega_{-1})^2W^3_{-3}\1\\
& \quad + \big(3 k^2 (k+2) (4445184 k^5+20157312 k^4+34479508 k^3+27362195 k^2+8585804 k-569072)\\
& \qquad\qquad /((64 k+107) (1424 k^2+3241 k+1542)) \omega_{-2}W^3_{-4}\1\\
& \quad - \big(8 k^2 (k+2) (6894208 k^5+54479264 k^4+153131647 k^3+192924011 k^2+105647314 k\\
& \qquad\qquad +17606072)/(3 (64 k+107) (1424 k^2+3241 k+1542)) \omega_{-1}W^3_{-5}\1\\
& \quad + \big(8 k^2 (1120640 k^6+26816288 k^5+152607907 k^4+385419551 k^3+501638434 k^2\\
& \qquad\qquad +336093968 k+94480832)/((64 k+107) (1424 k^2+3241 k+1542)) W^3_{-7}\1\\
& \quad - \big((16 k+17) (64 k+107)/(k (1424 k^2+3241 k+1542)) (W^3_{-1})^3\1\\
& \quad + \big(4 (k+2) (358 k+559)/(1424 k^2+3241 k+1542) \omega_{-1}W^3_{-1}W^4_{-1}\1\\
& \quad + \big(8 (449 k^2+1392 k+1075)/(1424 k^2+3241 k+1542) W^3_{-3}W^4_{-1}\1\\
& \quad + \big(112 k (k+2) (9 k^2+32 k+31)/(1424 k^2+3241 k+1542) \omega_{-3}W^5_{-1}\1\\
& \quad + \big(112 k (k+2)^2 (3 k+4)/(1424 k^2+3241 k+1542) (\omega_{-1})^2W^5_{-1}\1\\
& \quad - \big(2 k (k+2) (192 k^2+229 k-400)/(5 (1424 k^2+3241 k+1542)) \omega_{-2}W^5_{-2}\1\\
& \quad - \big(16 k (k+2) (16 k+17) (1686 k^2+3815 k+1786)\\
& \qquad\qquad /(5 (64 k+107) (1424 k^2+3241 k+1542)) \omega_{-1}W^5_{-3}\1\\
& \quad - \big(96 k (16 k+17) (2687 k^3+10292 k^2+12927 k+5342)\\
& \qquad\qquad /(5 (64 k+107) (1424 k^2+3241 k+1542)) W^5_{-5}\1\\
& \quad - \big(4/(k (1424 k^2+3241 k+1542)) W^4_{-1}W^5_{-1}\1.
\end{split}
\end{equation*}
}

\section{$\bu^0$, $\bu^1$, $\bu^2$ and $\bu^3$ in Case $k=5$}
\label{App:Case-k-5-singular-vec}

We express $\bu^0 = f(0)^{k+1}e(-1)^{k+1}\1$ and $\bu^r = (W^3_1)^r \bu^0$,
$r=1,2,3$ as linear combinations of the vectors of normal form in the case
$k = 5$.

{\footnotesize
\begin{equation*}
\begin{split}
\bu^0 &=
-(56260915200/97)\omega_{-5}\1 -(47822745600/97)\omega_{-3}\omega_{-1}\1
+ (43180603200/97)(\omega_{-2})^2\1\\
& \quad + (33230937600/97)(\omega_{-1})^3\1 -(4032/5)(W^3_{-1})^2\1
+ (550368/97)\omega_{-1}W^4_{-1}\1 + (340704/97)W^4_{-3}\1,
\end{split}
\end{equation*}
}

{\footnotesize
\begin{equation*}
\begin{split}
\bu^1 &=
-(17721088761600/5917)\omega_{-3}W^3_{-1}\1
+ (13262835801600/5917)(\omega_{-1})^2W^3_{-1}\1\\
& \quad + (221863017600/61)\omega_{-2}W^3_{-2}\1
- (365470963200/97)\omega_{-1}W^3_{-3}\1
+ (21001925203200/5917)W^3_{-5}\1\\
& \quad - (2122848/97)W^3_{-1}W^4_{-1}\1
- (89631360/61)\omega_{-1}W^5_{-1}\1
- (38413440/61)W^5_{-3}\1,
\end{split}
\end{equation*}
}

{\footnotesize
\begin{equation*}
\begin{split}
\bu^2 &=
(8181452462686782123600000/9757133)\omega_{-7}\1
+ (8868381288151420627200000/9757133)\omega_{-5}\omega_{-1}\1\\
& \quad - (5147471345450314255200000/9757133)\omega_{-4}\omega_{-2}\1
+ (23321410696693972800000/9757133)(\omega_{-3})^2\1\\
& \quad + (47380877265410942400000/159953)\omega_{-3}(\omega_{-1})^2\1
- (41194303604229799800000/159953)(\omega_{-2})^2\omega_{-1}\1\\
& \quad - (32478871712964566400000/159953)(\omega_{-1})^4\1
+ (498585049704000/1037)\omega_{-1}(W^3_{-1})^2\1\\
& \quad - (42034377168000/1037)W^3_{-3}W^3_{-1}\1
- (8566112126376000/159953)\omega_{-3}W^4_{-1}\1\\
& \quad - (524887446958560000/159953)(\omega_{-1})^2W^4_{-1}\1
+ (30178345962618000/159953)\omega_{-2}W^4_{-2}\1\\
& \quad - (340143285584592000/159953)\omega_{-1}W^4_{-3}\1
+ (357554169263088000/9757133)W^4_{-5}\1\\
& \quad + (89784495360/159953)(W^4_{-1})^2\1
+ (13751156160/1037)W^3_{-1}W^5_{-1}\1,
\end{split}
\end{equation*}
}

{\footnotesize
\begin{equation*}
\begin{split}
\bu^3 &=
(633349572703577384054400000/45093457)\omega_{-5}W^3_{-1}\1\\
& \quad - (304333657131610010822400000/45093457)\omega_{-3}\omega_{-1}W^3_{-1}\1\\
& \quad + (1141529140148275607400000/464881)(\omega_{-2})^2W^3_{-1}\1\\
& \quad + (69658304251736590588800000/45093457)(\omega_{-1})^3W^3_{-1}\1\\
& \quad + (175753574599043599200000/464881)\omega_{-4}W^3_{-2}\1\\
& \quad - (620956662666585739200000/464881)\omega_{-2}\omega_{-1}W^3_{-2}\1\\
& \quad - (5158511194620039076800000/45093457)\omega_{-3}W^3_{-3}\1\\
& \quad + (820496583733854986582400000/45093457)(\omega_{-1})^2W^3_{-3}\1\\
& \quad + (15005173408252695423000000/464881)\omega_{-2}W^3_{-4}\1\\
& \quad - (2704800801881903228784000000/45093457)\omega_{-1}W^3_{-5}\1\\
& \quad + (3141272873181195084427200000/45093457)W^3_{-7}\1
- (63906101745998400/7621)(W^3_{-1})^3\1\\
& \quad + (58024030066093728000/739237)\omega_{-1}W^3_{-1}W^4_{-1}\1
+ (70446353688003384000/739237)W^3_{-3}W^4_{-1}\1\\
& \quad + (29475630099262095840000/45093457)\omega_{-3}W^5_{-1}\1
+ (57072561118296796800000/45093457)(\omega_{-1})^2W^5_{-1}\1\\
& \quad - (5862027033141119568000/45093457)\omega_{-2}W^5_{-2}\1
- (82578847924067040000/464881)\omega_{-1}W^5_{-3}\1\\
& \quad - (2748731861102520384000/464881)W^5_{-5}\1
- (434544013612800/739237)W^4_{-1}W^5_{-1}\1.
\end{split}
\end{equation*}
}

\end{document}